\documentclass[12pt]{iopart}
\usepackage{amsthm,amssymb}
\usepackage[active]{srcltx}
\newtheorem{theorem}{Theorem}
\newtheorem{lemma}[theorem]{Lemma}

\newtheorem{remark}[theorem]{Remark}
\usepackage{mathptmx}       
\usepackage{helvet}         
\usepackage{courier}        
\usepackage{type1cm}        
\usepackage{makeidx}         
\usepackage{graphicx}        
\usepackage{multicol}        
\usepackage[bottom]{footmisc}
\def\be#1{\begin{equation}\label{#1}}
\def\ee{\end{equation}}
 
\def\be#1{\begin{equation}\label{#1}}
\usepackage[utf8]{inputenc}
 \newcommand{\s}[1]{\left(#1\right)}
 
 \newcommand{\n}[1]{\left\|#1\right\|}
 
 \renewcommand{\a}[1]{\left|#1\right|}

\makeindex            
\begin{document}
\title [Heuristic parameter choice in Tikhonov regularization]{Q-curve 
and area rules for choosing heuristic parameter in Tikhonov regularization}
\author{Toomas Raus and Uno H\"amarik}
\address{Institute of  Mathematics and Statistics, University of Tartu, Estonia}
\ead{toomas.raus@ut.ee, uno.hamarik@ut.ee}
\begin{abstract}
We consider choice of the regularization parameter in Tikhonov method if the noise level of the data is unknown.  One of the best rules for the heuristic parameter choice is the quasi-optimality criterion where the parameter is chosen as the global minimizer of the quasi-optimality function. In some problems this rule fails.
We prove that one of the local minimizers of the quasi-optimality function is always a good regularization parameter. For choice of the proper local minimizer we propose to construct 
the Q-curve which is the analogue of the L-curve, but on x-axis we use modified discrepancy instead of discrepancy and on the y-axis the quasi-optimality function instead of 
the norm of the approximate solution. In area rule we  choose for the regularization parameter such local minimizer of the quasi-optimality function for which the area of polygon, connecting on Q-curve this minimum point with certain maximum points, is maximal. We also provide a posteriori error estimates of the approximate solution, which allows to check the reliability of parameter chosen heuristically. Numerical experiments on extensive set of test problems confirm that the proposed rules give  much better results than  previous heuristic rules.  Results of proposed rules are comparable with results of the discrepancy principle and the monotone error rule, if last two rules use the exact noise level. 
\end{abstract}
\section{Introduction}

Let  $A\in{\cal L}(H,F)$ be a linear bounded operator between real Hilbert spaces $H$, $F$. We are interested in finding the minimum norm solution $u_*$ of the equation
\begin{equation}\label{mitteek_1}
Au=f_*,\qquad f_*\in {\cal R}(A),
\end{equation}
where noisy data $f\in F$ are given instead of the exact data $f_*$.
The range ${\cal R}(A)$ may be non-closed and the kernel ${\cal N}(A)$ may be non-trivial, so in general this problem is ill-posed. 
We consider solution of the problem $Au=f$ by 
Tikhonov method (see \cite{EHN, VaiVer}) where regularized solutions in cases of exact and inexact data have corresponding forms
\[
u^+_{\alpha}=\s{\alpha I+A^*A}^{-1}A^*f_*, \qquad u_{\alpha}=\s{\alpha I+A^*A}^{-1}A^*f
\]
and $\alpha>0$ is the regularization parameter. 
Using the well-known estimate $\n{u_{\alpha}-u^+_{\alpha}} \leq \frac12 {\alpha}^{-1/2}\n{f-f_*}$ (see \cite{EHN, VaiVer}) and notations
\be{e1}
e(\alpha):=\n{u_{\alpha}-u_*}, \qquad  e_1(\alpha):=\n{u^+_{\alpha}-u_*}+\n{u_{\alpha}-u^+_{\alpha}},
\ee
\[e_2(\alpha,\n{f-f_*}):=\n{u^+_{\alpha}-u_*}+\frac{1}{2 \sqrt \alpha} \n{f-f_*},\]
we have the error estimates
\be{err-est}
e(\alpha) \leq e_1(\alpha) \leq e_2(\alpha,\n{f-f_*}).
\ee 
We consider choice of the regularization parameter if the noise level for $\n{f-f_*}$ is unknown. The parameter choice rules which do not use the noise level information are called heuristic rules. Well known heuristic rules are the quasi-optimality criterion \cite{BaKi8, HPR09, Ki, Ki13, KiNe, Ne, TGK}, L-curve rule \cite{Ha92, Ha94}, GCV-rule \cite{GHW}, Hanke-Raus rule \cite{HR96}, Reginska's rule \cite{Re}; about other rules see \cite{HoReRo, Pa}. 
Heuristic rules are numerically compared in \cite{BaLu, HPR09, HoReRo, Pa}. The heuristic rules give good results in many problems, but it is not possible to construct heuristic rule guaranteeing convergence $\n{u_{\alpha}-u_*} \to 0$ as the noise level goes to zero (see \cite{Bak}). All heuristic rules may fail in some problems and without additional information about the solution, it is difficult to decide, is the obtained parameter reliable or not.

In the quasi-optimality criterion parameter $\alpha$ is chosen as the global minimizer of the function $\psi_{\mathrm{Q}}(\alpha)=\alpha\n{\frac{du_\alpha}{d\alpha}}$ on certain interval $[\alpha_N, \alpha_0]$. We propose to choose parameter from the 
set $L_{\mathrm{min}}$ of 
local minimizers of this function from certain set $\Omega$ of parameters.

We will call the parameter $\alpha_{\mathrm{R}}$ in arbitrary rule R as pseudooptimal, if 
\[\n{u_{\alpha_{\mathrm{R}}}-u_*} \leq c \, \min_{\alpha>0} e_1(\alpha)\] 
with relatively small constant $c$
and we show that at least one 
parameter from set $L_{\mathrm{min}}$
has this property. 
For the choice of proper parameter from the 
set $L_{\mathrm{min}}$ 
some algorithms were proposed in \cite{RH18}, in the current work we propose other algorithms. 
We propose to construct Q-curve which is the analogue of the L-curve \cite{Ha92}, but on x-axis we use modified discrepancy instead of discrepancy and on the y-axis the function $\psi_{\mathrm{Q}}(\alpha)$ instead of 
$\|u_\alpha\|.$
For finding proper local minimizer of the function $\psi_{\mathrm{Q}}(\alpha)$ 
we propose the area rules on the Q-curve. The idea of proposed rules is that we form for every minimizer of the function $\psi_{\mathrm{Q}}(\alpha)$ certain function which approximates the error of the approximate solution and has one minimizer; we choose for the regularization parameter such local minimizer of $\psi_{\mathrm{Q}}(\alpha)$ for which the area of polygon, connecting this minimum point with certain maximum points, is maximal.

The plan of this paper is as follows. In Section 2 we consider known rules for choice of the regularization parameter, both in case of known and unknown noise level. In Section 3 we 
prove that the set $L_{\mathrm{min}}$ contains at least one pseudooptimal parameter. In Section 4 information about used test problems (mainly from \cite{Ha94, BRS}, but also from \cite{Bk, Gr, Ha94, IndRam, Was}) and numerical experiments is given.
In Section 5 we consider the Q-curve and area rule, in Section 6 further developments of the area rule. These algorithms are also illustrated by results of numerical experiments.

\section{Rules for the choice of the regularization parameter}

\subsection{Parameter choice in the case of known noise level}
In case of known noise level $\delta, \n{f-f_*} \leq \delta$ we use one of so-called  $\delta$-rules, where certain functional $d(\alpha)$ and constant $b \geq b_0$ ($b_0$ depends on $d(\alpha)$) is chosen and such regularization parameter $\alpha(\delta)$ is chosen which satisfies 
$d(\alpha)= b \delta. $ 

1) Discrepancy principle (DP) \cite{Mo, VaiVer}: \[d_{\mathrm{D}}(\alpha):=\n{Au_{\alpha}-f} = b \delta,  \quad b \geq 1.\] 

2) Modified discrepancy principle (Raus-Gfrerer rule) \cite{Gf, Ra85}: 
\[d_{\mathrm{MD}}(\alpha):=\n{B_{\alpha}\s{Au_{\alpha}-f}} = b \delta, \quad B_{\alpha}:=\alpha^{1/2} \s{\alpha I+AA^*}^{-1/2}, \quad b \geq 1.\]
 
3) Monotone error rule (ME-rule) \cite{HKPRT, TaHa}:
\[d_{\mathrm{ME}}(\alpha):=\frac{\n{B_{\alpha}\s{Au_{\alpha}-f}}^2}{\n{B^2_{\alpha}\s{Au_{\alpha}-f}}} =  \delta.\]
The name of this rule is justified by the fact that
the chosen parameter $\alpha_{\mathrm{ME}}$ satisfies
\[ \frac{\mathrm{d}}{\mathrm{d}\alpha} \|u_\alpha - u_* \| > 0
 \quad \forall \alpha > \alpha_{\mathrm{ME}} \,.\]
Therefore $\alpha_{\mathrm{ME}} \geq \alpha_{\mathrm{opt}}:=\mbox{argmin} \|u_\alpha -u_*\|$.
 
4) Monotone error rule with post-estimation (MEe-rule) \cite{HPR09, HPR11, HPR12, Pa, RH09a}. The inequality $\alpha_{\mathrm{ME}} \geq \alpha_{\mathrm{opt}}$ suggests to use somewhat smaller parameter than $\alpha_{\mathrm{ME}}$. Extensive numerical experiments suggest to compute $\alpha_{\mathrm{ME}}$ and to use the post-estimated parameter $\alpha_{\mathrm{MEe}} :=0.4 \alpha_{\mathrm{ME}}$. Then typically $\|u_{\alpha_{\mathrm{MEe}}} - u_*\| / \|u_{\alpha_{\mathrm{ME}}} - u_*\| \in (0.7,0.9)$.
If the exact noise level is known, this MEe-rule gives typically the best results from all $\delta$-rules. 

5) Rule R1 \cite {Ra92}. Let $b > \frac{2}{3\sqrt{3}}$. Choose $\alpha(\delta)$  as the smallest solution of the equation \[d_{\mathrm{R1}}(\alpha(\delta)): =\alpha^{-1/2}\n{A^*B^2_{\alpha}\s{Au_{\alpha}-f}}=b \delta.\]
Note that this equation can be rewritten using the 2-iterated Tikhonov approximation $u_{2,\alpha}$:
\begin{equation}\label{MD}
B^2_{\alpha}\s{Au_{\alpha}-f}=Au_{2,\alpha}-f, \qquad u_{2,\alpha}=\s{\alpha I+A^*A}^{-1}(\alpha u_\alpha+A^*f).
\end{equation}

The last four rules are weakly quasioptimal rules  (see \cite{RH07}) for Tikhonov method:
if $\n{f-f_*} \leq \delta$, then
$
\n{u_{\alpha(\delta)}-u_*} \leq C(b) \inf_{\alpha>0} e_2(\alpha,\delta)
$ (see (\ref{err-est})).
The rules for the parameter choice in case of approximately given noise level are proposed and analyzed in \cite{HPR11, HPR12, Pa, RH09a}.

\subsection {Parameter choice in the case of unknown noise level}

 A classical heuristic rule is the quasi-optimality criterion. In Tikhonov method it chooses  
$\alpha=\alpha_{\mathrm{Q}}$ or $\alpha=\alpha_{\mathrm{QD}}$ as the global minimizer of corresponding functions
\begin{eqnarray} \label{qf}
\psi_{\mathrm{Q}}(\alpha)=\alpha\n{\frac{du_\alpha}{d\alpha}}={\alpha}^{-1}\n{A^*B^2_{\alpha}\s{Au_{\alpha}-f}}=\alpha \|A^*\s{\alpha I+AA^*}^{-2}f\|,  \\
\psi_{\mathrm{QD}}(\alpha)=\s{1-q}^{-1}\n{u_{\alpha}-u_{q\alpha}}, \quad 0<q<1.  \nonumber
\end{eqnarray}
The Hanke-Raus rule finds 
parameter  $\alpha=\alpha_{HR}$ as the global minimizer of the function
\[
\psi_{\mathrm{HR}}(\alpha)={\alpha}^{-1/2}\n{B_{\alpha}\s{Au_{\alpha}-f}}.
\]
In practice often L-curve is used. L-curve is log-log-plot of $\n{u_{\alpha}}$ versus $\n{Au_{\alpha}-f}$. 
The points $\s{\n{Au_{\alpha}-f},\n{u_{\alpha}}}$ have often shape similar to the letter L and parameter $\alpha_L$ which corresponds to the "corner point" is often a good parameter. In the literature several concrete rules for choice of the 'corner point'  are proposed. 
In \cite{Re}
parameter is chosen as the global minimizer of the function
 \[
\psi_{\mathrm{RE}}(\alpha)=\n{Au_{\alpha}-f}\n{u_{\alpha}}^{\tau}, \quad \tau \geq 1.
\]
(below we 
use this rule with 
$\tau=1$).  
Another rule for choice of the corner point is the maximum curvature method  (\cite{Ha98,CRS}), where such parameter $\alpha$ is chosen for which the curvature of the L-curve as the function \[\psi_{\mathrm{MC}}(\alpha)=2\frac{\hat \rho' \hat \xi ''-\hat \rho'' \hat \xi '}{((\hat \rho')^2+ (\hat \xi ')^2) ^{3/2}}\] is maximal. Here $\hat \rho',  \hat \xi ', \hat \rho'' ,  \hat \xi ''$ 
are first and second order derivatives of functions 
 $\log d_{\mathrm{D}}(\alpha)$ and $\log \n{u_{\alpha}}$. 
 
We propose also a new heuristic rule,
where the global minimizer of the function
\begin{equation} \label{WQ}
\psi_{\mathrm{WQ}}(\alpha)=d_{\mathrm{MD}}(\alpha) \psi_{\mathrm{Q}}(\alpha)
\end{equation} 
is chosen for the parameter. We call this rule as the weighted quasioptimality criterion.

In the following we will find the regularization parameter from the set of parameters
\be{omega}
\Omega=\left\{\alpha_j: \alpha_j=q\alpha_{j-1}, \quad j=1,2,...,N, \quad 0<q<1  \right\}
\ee
where $\alpha_0, q, \alpha_N$ are given. 
 If in the discretized problem the minimal eigenvalue $\lambda_{\mathrm{min}}$ of the matrix
 $A^TA$ is larger than $\alpha_N$, the heuristic rules above often choose parameter
 $\alpha_N$, which is generally not a good parameter. The works \cite{Ki13, KiNe, Ne} propose to search the global minimum of the function  $\psi_{\mathrm{Q}}(\alpha)$ in the interval $[\max{\s{\alpha_N,\lambda_{\mathrm{min}}}},\alpha_0]$. 

We say that the discretized problem
$Au=f$ does not need regularization if
\[e_1(\lambda_{\mathrm{min}}) \leq 2 \min_{\alpha \in \Omega, \ \alpha \geq \lambda_{\mathrm{min}}} e_1(\alpha).\] 
If  the discretized problem does not need regularization then $\alpha'=0$ or $\alpha'=\alpha_N$ is the proper parameter while 
for $\alpha' \leq \lambda_{min}$ we have 
\begin{eqnarray*}
 \n{u_{\alpha'}-u_*} \leq e_1(\alpha') =  \n{u^+_{\alpha'}-u_*}+\n{(\alpha' I+A^*A)^{-1}A^*(f-f_*)} \leq \\
\fl \qquad \n{u^+_{\lambda_{min}}-u_*} + 2\n{(\lambda_{min}I+A^*A)^{-1}A^*(f-f_*)} \leq 2 e_1(\lambda_{min}) \leq 
 4 \min_{\alpha \in \Omega,  \alpha \geq \lambda_{min}} e_1(\alpha).
\end{eqnarray*}
Searching the parameter from the interval $[\max{\s{\alpha_N,\lambda_{\mathrm{min}}}},\alpha_0]$ means the a priori assumption that the discretized problem needs regularization. Note that if $\lambda_{\mathrm{min}} > \alpha_N$, then in general it is not possible to decide (without additional information about solution or about noise of the data), whether the discretized problem needs regularization or not. 

\section {Local minimum points of the function  $\psi_{\mathrm{Q}}(\alpha)$}

In the following we investigate the function
 $\psi_{\mathrm{Q}}(\alpha)$ in (\ref{qf}) and show that at least one local minimizer of this function is the pseudooptimal parameter. 
We need some preliminary results. 

\begin{lemma}
The functions  $\psi_{\mathrm{Q}}(\alpha)$, $\psi_{\mathrm{QD}}(\alpha)$ satisfy for each  $\alpha > 0$ the estimates 
\begin{equation} \label{eq3}
\psi_{\mathrm{Q}}(\alpha) \leq e_1(\alpha),
\end{equation}
\begin{equation} \label{eq3a}
\psi_{\mathrm{QD}}(\alpha) \leq q^{-1} e_1(\alpha),
\end{equation}
\begin{equation*} 
\psi_{\mathrm{Q}}(\alpha) \leq \psi_{\mathrm{QD}}(\alpha) \leq q^{-1}\psi_{\mathrm{Q}}(q\alpha).
\end{equation*}
\end{lemma}
\begin{proof}
Using relations $f=Au_*+A^*(f-f_*)$,
\[ {u_{\alpha}-u_{q\alpha}=
\s{q-1}\alpha \s{\alpha I+A^*A}^{-1}}\s{q\alpha I+A^*A}^{-1} A^*f, \]
 \[\n{A^*A\s{\alpha I+A^*A}^{-1}} \leq 1, \quad \alpha \n{\s{\alpha I+A^*A}^{-1}} \leq 1 \]
we have
\begin{eqnarray*}
\psi_{\mathrm{Q}}(\alpha)&=&
\alpha \|A^*\s{\alpha I+AA^*}^{-2}f\|= 
\alpha \|(\alpha I+A^*A)^{-2}A^*f\| \\
&\leq&
\alpha \n{A^*A\s{\alpha I+A^*A}^{-2}u_*}+\alpha \n{\s{\alpha I+A^*A} ^{-2}A^*(f-f_*)}\\
&\leq&
\alpha \| (\alpha I+A^*A)^{-1}u_*\| +\|(\alpha I+A^*A) ^{-1}A^*(f-f_*)\|=  e_1(\alpha),\\ 
\psi_{\mathrm{QD}}(\alpha) &\leq&
\alpha \n{A^*A(q\alpha I+A^*A)^{-1})(\alpha I+A^*A)^{-1})u_*}\\
& &+\alpha \n{(q\alpha I+A^*A)^{-1})(\alpha I+A^*A)^{-1})A^*(f-f_*)} \leq q^{-1} e_1(\alpha),\\
\psi_{\mathrm{Q}}(\alpha)&=&\alpha  \n{\s{\alpha I+A^*A}^{-2}A^*f} \leq \alpha  \n{\s{\alpha I+A^*A}^{-1} \s{q\alpha I+A^*A}^{-1} A^*f} \\
&=& \psi_{\mathrm{QD}}(\alpha) \leq \alpha  \n{\s{q\alpha I+A^*A}^{-2}A^*f}=q^{-1}\psi_{\mathrm{Q}}(q\alpha).
\end{eqnarray*}
 \end{proof}
\begin{remark} Note that $\lim_{\alpha \rightarrow \infty} \psi_{\mathrm{Q}}(\alpha) =0 $, but  $\lim_{\alpha \rightarrow \infty} e_1(\alpha) =\n{u_*}$. 
Therefore in the case of too large $\alpha_0$ this $\alpha_0$ may be global (or local) minimizer of the function $\psi_{\mathrm{Q}}(\alpha)$. We recommend to take 
$\alpha_0=c\n{A^*A}, \ c \leq 1$ or to minimize the function  $\tilde{\psi}_{\mathrm{Q}}(\alpha):=(1+\alpha / \n{A^*A}) \psi_{\mathrm{Q}}(\alpha)$
instead of $\psi_{\mathrm{Q}}(\alpha)$. Due to limit
$\lim_{\alpha \rightarrow 0} (1+\alpha / \n{A^*A})=1$ the function $\tilde{\psi}_{\mathrm{Q}}(\alpha)$ approximately satisfies  (\ref{eq3}) for small $\alpha$.
\end{remark}

In the following we define the local minimum points of the function  $\psi_{\mathrm{Q}}(\alpha) $ on the set $\Omega$ (see (\ref{omega})).
We say that the parameter $\alpha_k ,\ 0 \leq k \leq N-1 $ is the local minimum point of the sequence $\psi_{\mathrm{Q}}(\alpha_k) $, if  $\psi_{\mathrm{Q}}(\alpha_k) <\psi_{\mathrm{Q}}(\alpha_{k+1})$ and in case $k>0 $ there exists index $j \geq 1$ such, that  $\psi_{\mathrm{Q}}(\alpha_k) =\psi_{\mathrm{Q}}(\alpha_{k-1}) =...=\psi_{\mathrm{Q}}(\alpha_{k-j+1}) <\psi_{\mathrm{Q}}(\alpha_{k-j})$.
 The parameter $\alpha_N$ is the local minimum point if there exists index $j \geq 1$ so, that  
\[\psi_{\mathrm{Q}}(\alpha_N) =\psi_{\mathrm{Q}}(\alpha_{N-1}) =...=\psi_{\mathrm{Q}}(\alpha_{N-j+1}) <\psi_{\mathrm{Q}}(\alpha_{N-j}).\]  
Denote the local minimum points by $m_k$, $k=1,\ldots,K$ ($K$ is the number of minimum points) and corresponding set by
$
L_{\mathrm{min}}=\left\{m_k: m_1>m_2>...>m_K\right\}.
$

The parameter $\alpha_k ,\ 0 < k < N $ is the local maximum point of the sequence $\psi_{\mathrm{Q}}(\alpha_k) $ if  $\psi_{\mathrm{Q}}(\alpha_k) >\psi_{\mathrm{Q}}(\alpha_{k+1})$ and there exists index $j \geq 1$ so, that  \[\psi_{\mathrm{Q}}(\alpha_k) =\psi_{\mathrm{Q}}(\alpha_{k-1})  =...=\psi_{\mathrm{Q}}(\alpha_{k-j+1}) >\psi_{\mathrm{Q}}(\alpha_{k-j}).\]
We denote by $M_k$ the local maximum point between the local minimum points  $m_{k+1}$ and  $m_k,\  1 \leq k \leq K-1$. Denote $M_0=\alpha_0$, $M_K=\alpha_N $. Then by the construction
\[
  M_K \leq m_K<M_{K-1}<\ldots<m_2<M_{1}<m_1 \leq M_0.
\]
\begin{theorem}
The following estimates hold for the local minimizers of the function $\psi_{\mathrm{Q}}(\alpha) $.
\begin{enumerate}
\item If $\alpha_0=\n{A^*A}$, $ \alpha_N =\alpha_0 \s{\frac{\n{f-f_*}}{\n{f_*}}}^2$, then 
\begin{equation} \label{eq61}
\fl \min_{\alpha \in L_{\mathrm{min}}} \n{u_{\alpha}-u_*} \leq  q^{-1}(1+2\max\{1,c_q\mid \ln \frac{\n{f-f_*}}{2\n{A}\n{u_*}}\mid\})\min_{\alpha>0}e_2(\alpha,\n{f-f_*}),
\end{equation}
where  $c_q:=\s{q^{-1}-1}/\ln{q^{-1}} \to 1\mbox{  if  } q \to 1$. 

Moreover, if $u_*=\a{A}^pv$, $\n{v}\leq\rho$, $p>0$, where $|A|:=(A^*A)^{1/2}$, then 
\begin{equation} \label{eq7}
\min_{\alpha \in L_{\mathrm{min}}} \n{u_{\alpha}-u_*} \leq   c_{p,q} \rho^\frac1{p+1}\a{\ln{\n{f-f_*}}}\n{f-f_*}^\frac p{p+1}, 0<p\leq 2.
\end{equation}
\item For arbitrary $\alpha_0, \alpha_N$ we have
\begin{equation} \label{eq6} 
 \min_{\alpha \in L_{\mathrm{min}}} \n{u_{\alpha}-u_*} \leq q^{-1}C \min_{\alpha_N \leq \alpha \leq \alpha_0} e_1(\alpha),   
\end{equation} 
\begin{equation*} 
\fl C:=1+\max_{1 \leq k \leq K} \ \max_{\alpha_j  \in \Omega, \, M_k \leq \alpha_j \leq M_{k-1}} T\s{m_k, \alpha_j} \leq 1+c_q\ln \s{\frac{\alpha_0}{\alpha_N}}, \ T(\alpha, \beta):=\frac{\n{u_{\alpha}-u_{\beta}}}{\psi_{\mathrm{Q}}(\beta)}.
\end{equation*} 
\end{enumerate}
\end{theorem}
\begin{proof}
For arbitrary parameters $\alpha \geq 0, \ \beta \geq 0$ the inequalities
\[\n{u_{\alpha}-u_*} \leq \n{u_{\alpha}-u_{\beta}} + \n{u_{\beta}-u_*} \leq  
T(\alpha, \beta)\psi_{\mathrm{Q}}(\beta)+e_1(\beta)\]
and (\ref{eq3}) lead to the estimate
\begin{equation} \label {eq4}
\n{u_{\alpha}-u_*} \leq \s{1+T(\alpha, \beta)}e_1(\beta).
\end{equation}
It is easy to see that 
\begin {equation}  \label {eq5}
\min_{\alpha_j \in \Omega} e_1(\alpha_{j}) \leq q^{-1}\min_{\alpha_N \leq \alpha \leq \alpha_0} e_1(\alpha),
\end {equation}
while in case $ q\alpha \leq \alpha' \leq \alpha$ we have $e_1\s{\alpha'} \leq  q^{-1} e_1\s{\alpha}$. 

Let ${\alpha}_{j*}=\alpha_0 q^{j*}$ be the global minimizer of the function $e_1(\alpha)$ on the set of the parameters $\Omega$. 
Then ${\alpha}_{j*} \in [M_k,M_{k-1}] $ for some $k, \ 1  \leq k \leq K$ and this $k$ defines index $m$ with $m_k=\alpha_m$.
From  (\ref{eq4})  we get the estimate
\begin {eqnarray*}
\fl \qquad \qquad \n{u_{m_k}-u_*} \leq \s{1+T(m_k, \alpha_{j*})}e_1(\alpha_{j*}) \leq  \s{1+\min_{M_k \leq \alpha_j \leq M_{k-1}}T(m_k, \alpha_{j})} \min_{\alpha_j \in \Omega} e_1(\alpha_{j})
\end {eqnarray*}
which together with (\ref{eq5}) gives also estimate (\ref{eq6}). 

Now we show that  $C \leq 1+c_q\ln \s{\frac{\alpha_0}{\alpha_N}}$. 
If  $m_k \leq \alpha_{j} \leq M_{k-1}$, Lemma 1 enables to estimate 
\[\n{u_{\alpha_m}-u_{\alpha_j}} \leq \sum_{j \leq i \leq m-1}\n{u_{i}-u_{i+1}} \leq q^{-1} (1-q) \sum_{j \leq i \leq m-1} \psi_{\mathrm{Q}}(\alpha_{i+1}),\]
\begin {eqnarray*}
 T(m_k, \alpha_{j}) &=&\frac{\n{u_{\alpha_m}-u_{\alpha_j}}}{\psi_{\mathrm{Q}}(\alpha_{j})} \leq
 q^{-1} (1-q)\sum_{j \leq i \leq m-1} \frac{\psi_{\mathrm{Q}}(\alpha_{i+1})}{\psi_{\mathrm{Q}}(\alpha_{j})} \\ &\leq&  (q^{-1}-1)(m-j) \leq  (q^{-1}-1) N= \frac{(q^{-1}-1)}{\ln{q^{-1}}}\ln{\frac{\alpha_0}{\alpha_N}}=c_q\ln{\frac{\alpha_0}{\alpha_N}}.
\end {eqnarray*}

If $M_k \leq \alpha_{j} \leq m_k$, then analogous estimation of  $T(m_k, \alpha_{j})$  gives the same result. 

Now we prove the estimate (\ref{eq61}). For the global minimum point $\alpha_{*}$ of the function $e_2(\alpha, \n{f-f_*})$ the inequality $\alpha_{*} \geq \alpha_N$ holds, while 
for $\alpha<\alpha_N$ we have 
\[e_2(\alpha^*) \leq \|u_*\|=\|f-f_*\|/(2\sqrt{\alpha_N}) \leq \|f-f_*\|/(2\sqrt{\alpha}) < e_2(\alpha).\]
In the case $\alpha_* \leq \alpha_0$ we get similarly as in the proof of estimate (\ref{eq6}) that
\[\min_{\alpha \in L_{\mathrm{min}}} \n{u_{\alpha}-u_*} \leq  q^{-1}(1+c_q\ln{\frac{\alpha_0}{\alpha_N})\min_{\alpha>0}e_2(\alpha,\n{f-f_*})};\]
due to $\ln{\frac{\alpha_0}{\alpha_N}}=\mid \ln \n{f-f_*}/{\n{f_*}} \mid$ 
the estimate (\ref{eq61}) holds. Consider the case  $\alpha_* > \alpha_0$. Then \[e_2(\alpha_*,\n{f-f_*})  \geq \n{u^+_{\alpha_0}-u_*} \geq \frac{\alpha_0}{\alpha_0+\n{A^*A}}\n{u_*} =  \frac{\n{u_*}}{2}.\] and for each local minimum point  $m_k, \alpha_N \leq m_k \leq \alpha_0$ the inequalities
\begin {eqnarray*}
\n{u_{m_k}-u_*} \leq e_2(m_k,\n{f-f_*}) \leq \n{u^+_{\alpha_0}-u_*}+0.5{\alpha_N}^{-1/2} \n{f-f_*} =\\ \n{u^+_{\alpha_0}-u_*}+\n{u_*} 
\leq 3 \n{u^+_{\alpha_0}-u_*} \leq 3 e_2(\alpha_*,\n{f-f_*}) 
\end {eqnarray*}
hold. Therefore the inequality (\ref{eq61}) holds also in this case.

For source-like solution $u_*=\a{A}^pv$, $\n{v}\leq\rho$, $p>0$ the error estimate
\begin {eqnarray*}
\min_{\alpha_N \leq \alpha \leq \alpha_0} e_1(\alpha) \leq c_p {\rho}^{1/(p+1)}{\n{f-f_*}}^{p/(p+1)},  0 < p \leq 2
\end {eqnarray*}
is well-known (see \cite{EHN, VaiVer})  and the estimate  (\ref{eq7}) follows immediately from (\ref{eq61}).
\end{proof}
\begin {remark}
Theorem 3 holds also in the case if the equation $Au=f_*$ has only the quasisolution, i.e. in the case $f_* \notin {\cal R}(A)$, $Qf_* \in {\cal R}(A)$,  where $Q$ is the orthoprojector $F \rightarrow \overline{{\cal R}(A)}$.
\end {remark}
\begin {remark}
The inequality (\ref{eq6}) holds also in the case if the noise of the $f$ is not finite but $\min_{\alpha_N \leq \alpha \leq \alpha_0} e_1(\alpha)$ is finite (this holds if $\n{A^*(f - f_*)}$ is finite).  
\end {remark}
\begin {remark}
Use of the inequality (\ref{eq3a}) enables to prove the analogue of Theorem 3 for set $L_{\mathrm{min}}$ of local minimizers of the function $\psi_{\mathrm{QD}}(\alpha)$: then the inequality (\ref{eq6}) holds, where  $T(\alpha,\beta)=q^{-1} \frac{\n{u_{\alpha}-u_{\beta}}}{\psi_{QD}(\beta)}$.  
\end {remark}

In choice of the regularization parameter we may exclude from the observation some local minimizers. It is natural to assume that $\alpha_N$ is so small that 
\begin{equation} \label{a-small}
d_{\mathrm{MD}}(\alpha_N) \leq (1+\epsilon) \n{f-f_*}
\end{equation}
 with small $\epsilon>0$. Then the following theorem holds.
\begin{theorem}
Let (\ref{a-small}) holds. Let $m_{k_0}$ be some local minimizer in  $L_{\mathrm{min}}$. Then
\[ \fl \qquad \min_{\alpha \in L_{\mathrm{min}}, \alpha \geq m_{k_0}} \n{u_{\alpha}-u_*} \leq \max\{q^{-1}C_1 \min_{\alpha \geq 0} e_1(\alpha), C_2(b,\epsilon) \min_{\alpha \geq 0}  e_2(\alpha,\n{f-f_*})\}, \]
where $b=d_{\mathrm{MD}}(m_{k_0})/d_{\mathrm{MD}}(\alpha_N) \geq 1, \quad
C_2(b,\epsilon):=b(1+\epsilon)+2$ and
\[C_1:=1+\max_{1 \leq k \leq k_0} \  \max_{\alpha_j  \in \Omega,\  M_k \leq \alpha_j \leq M_{k-1}} T\s{m_k, \alpha_j} \leq 1+c_q\ln \s{\frac{\alpha_0}{m_{k_0}}}.\] 
\end{theorem}
\begin{proof} Let $\alpha^*_i, \ i=1,2$  be global minimizers of the functions $e_1(\alpha)$ and $e_2(\alpha, \n{f-f_*})$  respectively. 
We consider separately 3 cases.
If $m_{k_0} \leq \alpha^*_{1}$ we get similarly to the proof of Theorem 3 the estimate 
\begin{equation} \label{eq14}
\min_{\alpha \in L_{\mathrm{min}}, \alpha \geq m_{k_0}} \n{u_{\alpha}-u_*} \leq q^{-1}C_1 \min_{\alpha_N \leq \alpha \leq \alpha_0} e_1(\alpha).
\end{equation}
If $\alpha^*_{1}  \leq m_{k_0} <\alpha^*_{2}$ we estimate 
\begin{equation} \label{eq15}
\fl \qquad  \quad \n{u_{m_{k_0}}-u_*} \leq \n{u^+_{\alpha^*_2}-u_*}+\frac{\n{f-f_*}}{2\sqrt{\alpha^*_1}} \leq \min_{\alpha \geq 0} e_2(\alpha, \n{f-f_*}) + \min_{\alpha \geq 0} e_1(\alpha).
\end{equation}
If $\alpha^*_{1}  \leq m_{k_0}$ and $\alpha^*_{2} \leq m_{k_0}$ we have 
\[\n{B_{m_{k_0}}\s{Au_{m_{k_0}}-f}} \leq b d_{\mathrm{MD}}(\alpha_N) \leq b(1+\epsilon)\n{f-f_*}\] 
and now we can prove analogically to the proof of the weak quasioptimality of the modified discrepancy principle 
(\cite{RH07}) that under assumption  $\alpha^*_{2} \leq m_{k_0}$ the error estimate
\begin{equation}  \label{eq16}
\n{u_{{m_{k_0}}}-u_*} \leq  C_2(b,\epsilon) \min_{\alpha \geq 0} e_2(\alpha, \n{f-f_*})
\end{equation}
holds. Now the assertion 1 of Theorem 7 follows from the inequalities
(\ref{eq14})-(\ref{eq16}).
\end{proof}

\section {On test problems and numerical experiments}
We made numerical experiments for local minimizers of the function $\psi_{\mathrm{Q}}(\alpha)$ using
three sets of test problems. 
The first set contains 10 well-known test problems from Regularization Toolbox \cite{Ha94} and the following 6 Fredholm integral equations of the first kind
(discretized by the midpoint quadrature formula )
\[\int_{a}^{b}K(t,s)u(s) ds=f(t), \quad  c \leq t \leq d.\] 
\begin{itemize}
\item \emph{groetsch1} \cite{Gr}:
$K(t,s)=\frac{t\exp{(-t^2/(4s))}}{2\sqrt{\pi}s^{3/2}}$, $\ 0 \leq s,t \leq 100$,
$\ u(s)=40+\\
5\cos{((100-s)/5)}+2.5\cos{(2(100-s)/2.5)}+1.25\cos{(4(100-s)/2)};$
\item \emph{groetsch2} \cite{Gr}:
$K(t,s)=\sum_{1 \leq k \leq 100}{\frac{\sin{(kt)}\sin{(ks)}}{k}}$, $\ 0 \leq s,t \leq \pi$,  
$\ u(s)=s(\pi-s);$
\item \emph{indram} \cite{IndRam} :
$K(t,s)=e^{-st}$,  $\ 0 \leq s,t \leq 1$, $\ u(s)=s$, $\ f(t)=\frac{1-(t+1)e^{-t}}{t^2};$
\item \emph{ursell} \cite{Ha94}:
$K(t,s)=\frac{1}{1+s+t}$, $\ 0 \leq s,t \leq 1$, $\ u(s)=s(1-s)$, $\ f(t)=\frac{3+2t}{2}+ \\(2+3t+t^2)\log{\s{\frac{1+t}{2+t}}};$
\item \emph{waswaz} \cite{Was}: 
$K(t,s)=\cos{(t-s)}$, $\ 0 \leq s,t \leq \pi, $
$\ u(s)=\cos{(s)}$, $\ f(t)=\frac{\pi}{2} \cos{(t)};$
\item \emph{baker} \cite{Bk}: $K(t,s)=e^{st}$, $\ 0 \leq s,t \leq 1$, $u(s)=e^{s}$, $\ f(t)=\frac{e^{t+1}-1}{t+1}.$
\end{itemize}
The second set of test problems are well-known problems 
from  \cite{BRS} :  \emph{gauss, hilbert, lotkin, moler, pascal, prolate}. As in  \cite{BRS}, we combined these six $n \times n$ matrices with 6 solution vectors $x_i = 1, x_i = i/n $, 
$x_i = ((i - [n/2])/[n/2])^2$, $ x_i = \sin{(2 \pi (i -1)/n)}$, $x_i = i/n +1/4 \sin{(2 \pi (i -1)/n)}$, $x_i = 0$ if $i \leq [n/2]$ and $x_i = 1$ if $i > [n/2]$. For getting the third set of test problems we combined the matrices of the first set of test problems with 6 solutions of the second set of test problems. 

Numerical experiments showed that performance of different rules depends essentially on eigenvalues of the matrix $A^T A$. We characterize these eigenvalues via three indicators: the value of minimal eigenvalue 
$\lambda_{\mathrm{min}}$, by the value $N_1$, showing number of eigenvalues less than $\alpha_N$ and by the value $\Lambda$, characterizing the density of location of eigenvalues on the interval
$[\max{(\alpha_N,\lambda_{\mathrm{min}})},1]$. More precisely, let the eigenvalues of the matrix  $A^TA$ be   $\lambda_1 \geq \lambda_2 \geq ... \geq \lambda_n=\lambda_{\mathrm{min}}$.
Then value of $\Lambda$ is found by the formula
$\Lambda=\max_{ \lambda_k>\max{(\alpha_N,\lambda_n)}} \lambda_k/\lambda_{k+1} $.
We characterize the smoothness of the solution  by value 
\[p1=\frac{\log \min_{\alpha} e_2(\alpha, \n{f-f_*})-\log\n{u_*}}{\log \n{f-f_*} - \log \n{f_*}},\] 
where $\n{f-f_*}=10^{-6}$.
Table 1 contains the results of characteristics of the matrix  $A^T A$  in case $n=100, \ \alpha_N=10^{-18}.$ 
\begin{table}
\caption{Characteristics of matrix $A^T A$ and the solution $u_*$}
\label{tab:1}       
\begin{tabular}{lcccc|lcccc}
\br
Problem  & $\lambda_{\mathrm{min}}$ & $N_1$ & $\Lambda$ & p1 &Problem  & $\lambda_{\mathrm{min}}$ & $N_1$ & $\Lambda$ & p1   \\
\mr
Baart	&5.2E-35&	92&	1665.7& 0.197 &Spikes	&1.3E-33&	89&	1529.3&  0.005 \\
Deriv2	&6.7E-09&	0&		16.0& 0.286	&Wing	&2.9E-37&	94&	9219.1& 0.057  \\
Foxgood	&9.0E-33&	85&	210.1&0.426	&Baker	&1.0E-33&	94&	9153.1&  0.498\\
Gravity	&1.6E-33&	68&	4.1&  0.403 &Ursell	&6.9E-34&	94&	3090.2&  0.143 \\
Heat	&5.5E-33&	3&	2.4E+20&	0.341&Indramm	&2.7E-33&	94&	9154.6&  0.395 \\
Ilaplace	&3.8E-33&	79	&16.1&	 0.211&Waswaz2 	&2.0E-34&98&	1.7E+30&  0.654\\
Phillips	&1.4E-13&	0&	9.4&0.471 &Groetsch1	&5.8E-33&	78&	11.2&  0.176 \\
Shaw & 2.3E-34&	85&	289.7&0.244	&Groetsch2	&1.0E-04&	0&  	4.0&  0.652 \\
\br
\end{tabular}
\end{table}

In all tests discretization parameters   
 $n \in \{60, 80, 100, 120, 140, 160, 180\}$ were used. We present the results of numerical experiments in tables for $n=100$.
Since the performance of rules generally depends on the smoothness $p$ of the exact solution in (\ref{mitteek_1}), we complemented the standard solutions $u_*$ of (now discrete) test problems with smoothened solutions 
$|A|^p u_*, p=2$ computing the right-hand side as $A(|A|^p u_*)$.
Results for $p=2$ are given in Table 7, in all other tables and figures $p=0$.
 After discretization all problems were scaled (normalized) in such a way that the norms of the operator and the right-hand side were 1. 
All norms here and in the text below are Euclidean norms. 
On the base of exact data $f_*$ we formed the noisy data $f$, where $\n{f-f_*}$ has values $10^{-1}, 10^{-2},...,10^{-6}$, noise $f-f_*$ has normal distribution and the components of the noise were uncorrelated. We generated 20 noise vectors and used these vectors in all problems. We search the regularization parameter from the set 
$\Omega$, where $\alpha_0=1, q=0.95 $ and $N$ is chosen so that  $\alpha_N \geq 10^{-18} > \alpha_{N+1}$.  
To guarantee that calculation errors do not influence essentially the numerical results, calculations were performed on geometrical sequence of decreasing $\alpha$-s and finished for largest $\alpha$ with $d_{\mathrm{MD}}(q\alpha)>d_{\mathrm{MD}}(\alpha)$, while theoretically the function $d_{\mathrm{MD}}(\alpha)$ is monotonically increasing. 
Actually this precautionary measure was needed only in problem \emph{groetsch2}, calculations on $\alpha>\alpha_N$ were finished only in this problem.
Since in model equations the exact solution is known, it is possible to find the regularization parameter $\alpha_*$, which gives the smallest error on the set $\Omega$ . 
For every rule R the error ratio 
\[E=\frac{\n{u_{\alpha_{\mathrm{R}}}-u_*}}{\n{u_{\alpha_*}-u_*}}= \frac{\n{u_{\alpha_{\mathrm{R}}}-u_*}}{\min_{\alpha \in \Omega} \n{u_{\alpha}-u_*}}\]
describes the performance of the rule R on this particular problem. To compare the rules or to present their properties, the following tables show averages and maximums of these error ratios over various parameters of the data set (problems, noise levels $\delta$).  We say that the heuristic rule fails if the error ratio $E>100$.
In addition to the error ratio $E$ we present in some cases also error ratios 
 \[E1= \frac{\n{u_{\alpha_{\mathrm{R}}}-u_*}}{\min_{\alpha \in \Omega} e_1(\alpha)}, \quad E2= \frac{\n{u_{\alpha_{\mathrm{R}}}-u_*}}{\min_{\alpha \in \Omega} e_2(\alpha)}.\]

\begin{table}
\caption{Results for the set $L_{\mathrm{min}}$}
\label{tab:2}       
\begin{tabular}{lccc|cc|cc|cc}
\br
Problem & ME & MEe & DP  & \multicolumn{2}{c|}{Best of $L_{\mathrm{min}}$}  &\multicolumn{2}{c|}{ $|L_{\mathrm{min}}|$} &\multicolumn{2}{c}{ Apost.  $C$ }\\
 & Aver E & Aver E & Aver E  & Aver E & Max E  & Aver & Max  & Aver & Max\\
\mr
Baart        & 1.43  & 1.32 & 1.37 & 1.23 & 2.51 & 6.91 & 8  & 3.19 & 3.72 \\
Deriv2        & 1.29 & 1.07 & 1.21 & 1.08 & 1.34 & 1.71 & 2  & 3.54 & 4.49 \\
Foxgood     & 1.98 & 1.42 & 1.34 & 1.47 & 6.19 & 3.63 & 6  & 3.72 & 4.16 \\
Gravity       & 1.40 & 1.13 & 1.16 & 1.13 & 1.83 & 1.64 & 3  & 3.71 & 4.15 \\
Heat          & 1.19 & 1.03 & 1.05 & 1.12 & 2.36 & 3.19 & 5  & 3.92 & 4.50 \\
Ilaplace      & 1.33 & 1.21 & 1.26 & 1.20 & 2.56 & 2.64 & 5  & 4.84 & 6.60 \\
Phillips       & 1.27 & 1.02 & 1.02 & 1.06 & 1.72 & 2.14 & 3  & 3.99 & 4.66 \\
Shaw         & 1.37 & 1.24 & 1.28 & 1.19 & 2.15 & 4.68 & 7  & 3.48 & 4.43 \\
Spikes        & 1.01 & 1.00 & 1.01 & 1.00 & 1.02 & 8.83 & 10 & 3.27 & 3.70 \\
Wing          & 1.16 & 1.13 & 1.15 & 1.09 & 1.38 & 5.20 & 6  & 3.07 & 3.72 \\
Baker         & 3.91 & 2.38 & 2.09 & 2.31 & 16.17 & 5.38 & 6  & 3.14 & 3.72 \\
Ursell         & 2.14 & 1.97 & 2.03 & 1.69 & 4.44 & 5.53 & 6   & 3.07 & 3.43 \\
Indramm    & 5.20 & 3.26 & 3.37 & 3.38 & 25.67 & 5.64 & 6  & 3.08 & 3.71 \\
Waswaz2    & 127.2 & 49.9 & 1.20 & 2.44 & 9.03 & 1.00 & 1   & 2.00 & 2.00 \\
Groetsch1 & 1.12 & 1.07 & 1.08 & 1.06 & 1.51 & 3.99 & 7   & 4.23 & 5.20 \\
Groetsch2 & 1.02 & 1.22 & 1.67 & 1.13 & 1.69 & 1.67 & 2   & 5.62 & 13.72 \\
\mr
Set 1         & 9.62 & 4.46 & 1.46 & 1.48 & 25.67 & 3.99 & 10 & 3.67 & 13.72 \\
Set 2         & 1.57 & 1.32 & 1.36 & 1.20 & 5.33 & 4.40 & 10 & 3.50 & 5.43 \\
Set 3         & 7.19 & 3.45 & 1.47 & 1.48 & 61.02 & 3.64 & 10 & 3.73 & 9.12 \\
\br
\end{tabular}
\end{table}

The results of numerical experiments for local minimizers $\alpha \in L_{\mathrm{min}}$ of the function
 $\psi_{\mathrm{Q}}(\alpha)$ are given in the Table 2. For comparison the results of $\delta$-rules with $\delta=\n{f-f_*}$ are presented in the columns 2-4. 
Columns 5 and 6 contain respectively the averages and maximums of error ratios $E$ for the best local minimizer $\alpha \in L_{\mathrm{min}}$. The results show that for many problems the Tikhonov approximation with the best local minimizer $\alpha \in L_{\mathrm{min}}$ is even more accurate than with the $\delta$-rules parameters $\alpha_{\mathrm{ME}}, \alpha_{\mathrm{MEe}}$ or $\alpha_{\mathrm{DP}}$. 
Tables 1, 2 show also that for rules ME and MEe the average error ratio $E$ may be relatively large for problems where $\Lambda$ is large and most of eigenvalues are smaller than $\alpha_N$, while in this case $\min_{\alpha \in \Omega} e(\alpha)$ may be essentially smaller than $\min_{\alpha \in \Omega} e_2(\alpha, \n{f_*-f})$. In these problems the discrepancy principle gives better parameter than ME and MEe rules.
Columns 7 and 8 contain the averages and maximums of cardinalities $|L_{\mathrm{min}}|$ of sets $L_{\mathrm{min}}$ (number of elements of these sets). Note that number of local minimizers depends on parameter
$q$ (for smaller $q$ the number of local minimizers is smaller) and on length of minimization interval determined by the parameters $\alpha_N$, $\alpha_0$. The number of local minimizers is smaller also for larger noise level. 
Columns 9 and 10 contain the averages and maximums of values of constant
$C$ in the a posteriori error estimate (\ref{eq6}). The value of $C$ and error estimate (\ref{eq6}) allow to assert, that in our test problems the choice of $\alpha$ as the best local minimizer in $L_{\mathrm{min}}$ guarantees that error of the Tikhonov approximation has the same order as $\min_{\alpha_N \leq \alpha \leq \alpha_0} e_1(\alpha)$. Note that over all test problems the maximum of error ratio $E1$ for the best local minimizer
in $L_{\mathrm{min}}$ and for the discrepancy principle were 1.93 and 9.90 respectively. This confirm the result of Theorem 3 that at least one minimizer of the function $\psi_{\mathrm{Q}}(\alpha)$ is a good regularization parameter.

\section {Q-curve and triangle area rule for choosing heuristic regularization parameter}



We showed in previous section that at least one local minimizer of the function $\psi_{\mathrm{Q}}(\alpha)$ is pseudooptimal parameter and we may omit small local minimizers $\alpha$, for which $d_{\mathrm{MD}}(\alpha)$ is only slightly larger than  $d_{\mathrm{MD}}(\alpha_N)$. We propose to construct for parameter choice the Q-curve
The Q-curve figure uses log-log scale with functions $d_{\mathrm{MD}}(\alpha)$ and $\psi_{\mathrm{Q}}(\alpha)$ on the $x$-axis and $y$-axis respectively. 
The Q-curve can be considered as the analogue of L-curve, where functions $Au_\alpha-f$ and $u_\alpha=-\alpha^{-1}A^*(Au_\alpha-f)$ are replaced by functions $B_\alpha(Au_\alpha-f)$ and 
$-\alpha^{-1}A^*B^2_\alpha (Au_\alpha-f)$ (see (\ref{MD})) respectively.
We denote $\tilde d_{\mathrm{MD}}(\alpha):=\log_{10} d_{\mathrm{MD}}(\alpha),$ $\tilde \psi_{\mathrm{Q}}(\alpha):=\log_{10} \psi_{\mathrm{Q}}(\alpha) $. For many problems the curve $(\tilde d_{\mathrm{MD}}(\alpha),\tilde \psi_{\mathrm{Q}}(\alpha) )$ (or a part of this) has the form of letter L or V and we choose the minimizer at the "corner" point of L or V. We use the common logarithm instead of natural logarithm, while then the Q-curve allows easier to estimate the supposed value of the noise level. 
On the figures 1-8 $n=100$  is used, on the figures 9, 10 $n=60$.
On the figures 1-4 the L-curves and Q-curves are compared for two problems, the global minimizer $\alpha_{opt}$ of the function $e_1(\alpha)$ is also presented.  Note that in problem  \emph{baart} $\lambda_{\mathrm{min}} < \alpha_N$ and in problem  \emph{deriv2} $\lambda_{\mathrm{min}} > \alpha_N$. 

\begin{figure}[!tbp]
  \begin{minipage}[b]{0.45\textwidth}
    \includegraphics[width=\textwidth]{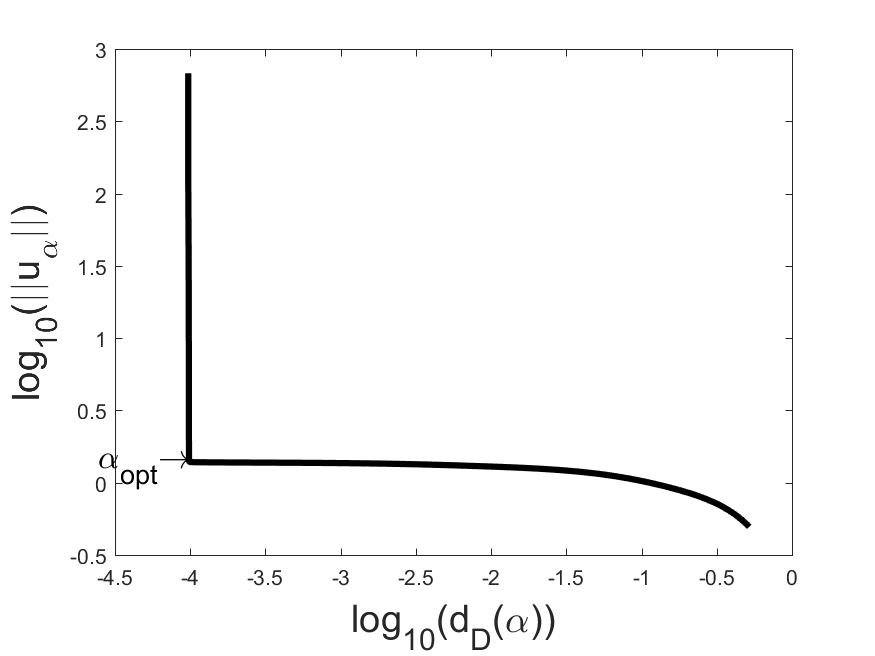}
    \caption{L-curve for \emph{baart}.}
  \end{minipage}
  \hfill
  \begin{minipage}[b]{0.45\textwidth}
    \includegraphics[width=\textwidth]{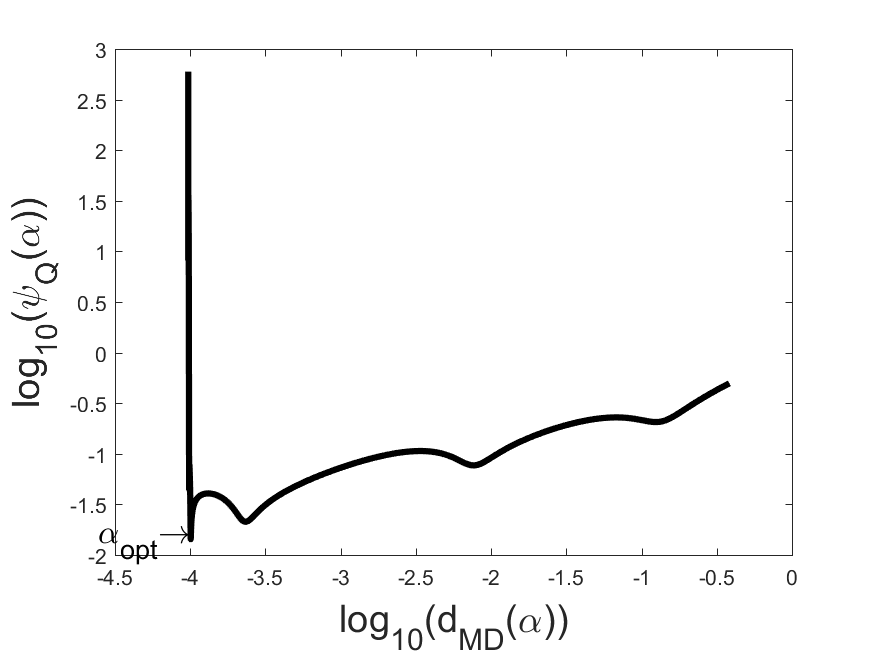}
    \caption{Q-curve for \emph{baart}.}
  \end{minipage}
\end{figure}

\begin{figure}[!tbp]
  \centering
  \begin{minipage}[b]{0.45\textwidth}
    \includegraphics[width=\textwidth]{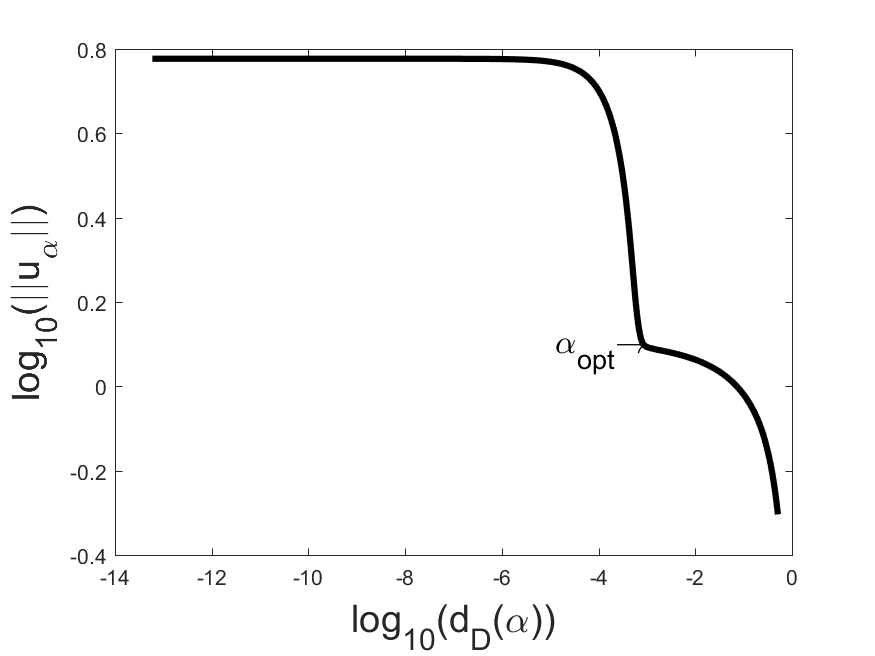}
    \caption{L-curve for \emph{deriv2}.}
  \end{minipage}
  \hfill
  \begin{minipage}[b]{0.45\textwidth}
    \includegraphics[width=\textwidth]{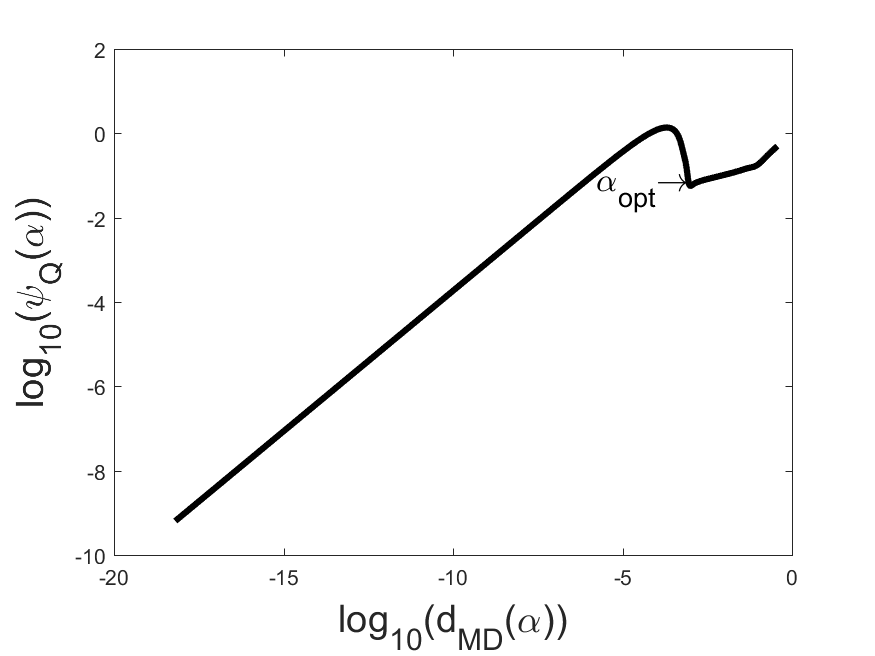}
    \caption{Q-curve for \emph{deriv2}.}
  \end{minipage}
\end{figure}

In most cases one can see on the Q-curve only one clear "corner" with one local minimizer. If the "corner" contains several local minimizers, we recommend to choose such local minimizer, for which the sum of coordinates of corresponding point on the Q-curve is minimal. If Q-curve has several "corners" we recommend to use the very right of them. 
Actually, it is useful to present in parameter choice besides figures  for every local minimizer $m_k$ of the function   $\psi_{\mathrm{Q}}(\alpha)$ also coordinates of point $(\tilde d_{\mathrm{MD}}(m_k),\tilde \psi_{\mathrm{Q}}(m_k) )$ and sums of coordinates. 

For finding proper local minimizer of the function $\psi_{\mathrm{Q}}(\alpha)$ we present now a rule which works well for all test problems from set 1.  The idea of rule is to search proper local minimizer $m_k$
constructing certain triangles on the Q-curve and finding which of them has the maximal area. For parameter $\alpha$ corresponds a point $P(\alpha)$ on the Q-curve with corresponding coordinates   $\s{\tilde d_{\mathrm{MD}}(\alpha),\tilde \psi_{\mathrm{Q}}(\alpha)}$.
For every local minimizer $m_k$ of the function $\psi_{\mathrm{Q}}(\alpha)$ corresponds a triangle $T(k,r(k),l(k))$ with vertices $P(m_k), P(M_{r(k)})$ and $P(M_{l(k)})$ on the Q-curve, where indices $r(k)$ and $l(k)$ correspond to the largest local maximums of the function $\psi_{\mathrm{Q}}(\alpha)$ on two sides of the local minimum $m_k$:
   \[\psi_{\mathrm{Q}}(M_{r(k)})=\max_{j < k} \psi_{\mathrm{Q}}(M_j), \qquad \psi_{\mathrm{Q}}(M_{l(k)})=\max_{j \geq k} \psi_{\mathrm{Q}}(M_j). \]

 \textbf{Triangle area rule (TA-rule).} We choose for the regularization parameter such local minimizer $m_k$ of the function $\psi_{\mathrm{Q}}(\alpha)$ for which the area of the triangle $T(k,r(k),l(k))$ is the largest.
 

\begin{figure}[!tbp]
  \centering
  \begin{minipage}[b]{0.45\textwidth}
    \includegraphics[width=\textwidth]{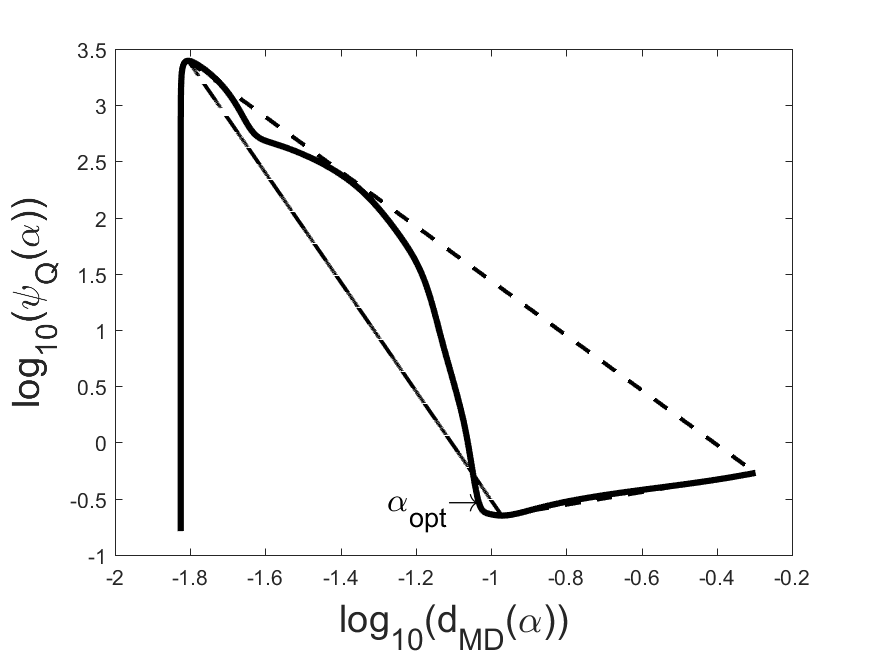}
    \caption{Q-curve in \emph{heat}.}
  \end{minipage}
  \hfill
  \begin{minipage}[b]{0.45\textwidth}
    \includegraphics[width=\textwidth]{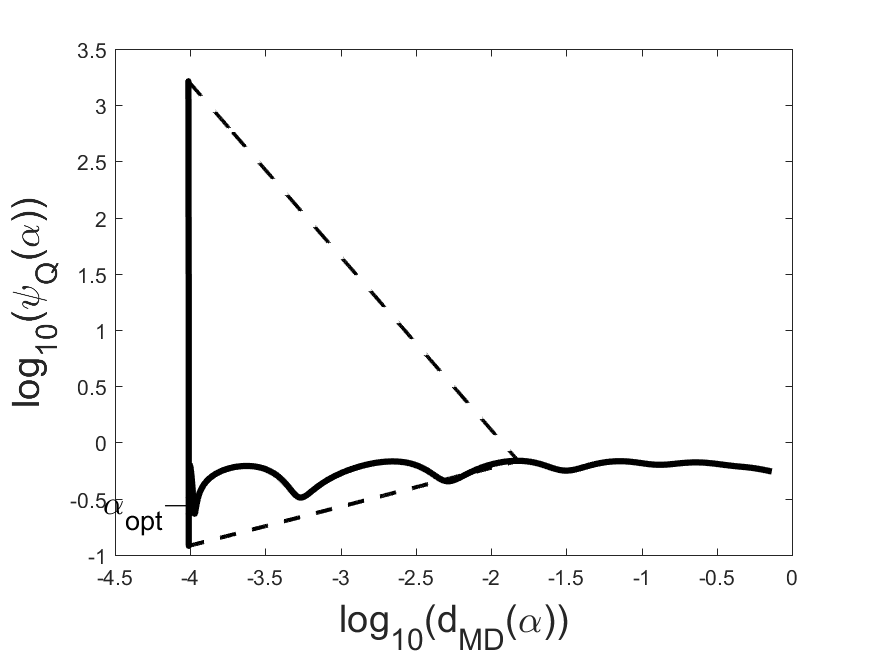}
    \caption{Q-curve in \emph{spikes}.}
  \end{minipage}
\end{figure}

\begin{figure}[!tbp]
  \centering
  \begin{minipage}[b]{0.45\textwidth}
    \includegraphics[width=\textwidth]{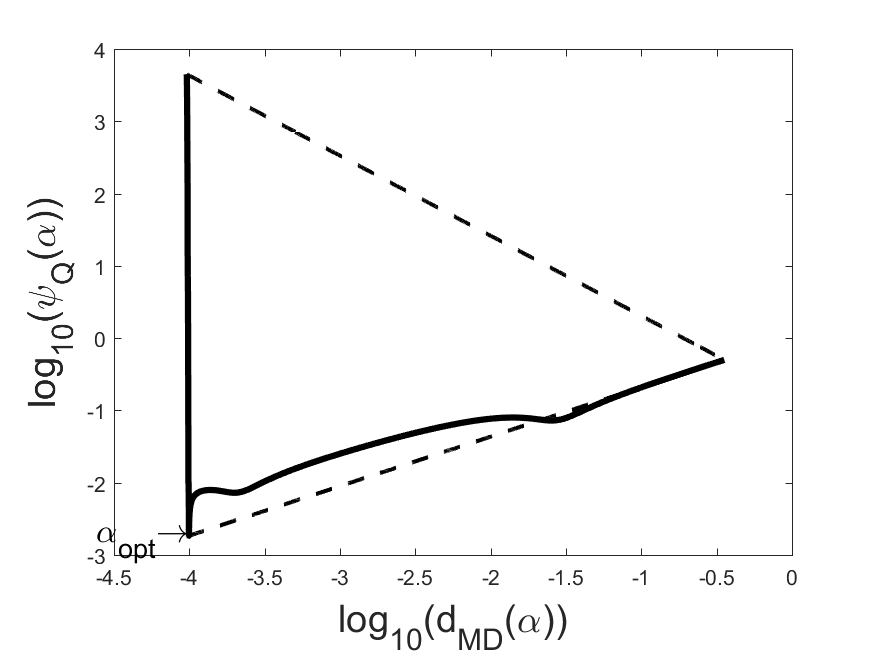}
    \caption{Q-curve in \emph{foxgood}.}
  \end{minipage}
  \hfill
  \begin{minipage}[b]{0.45\textwidth}
    \includegraphics[width=\textwidth]{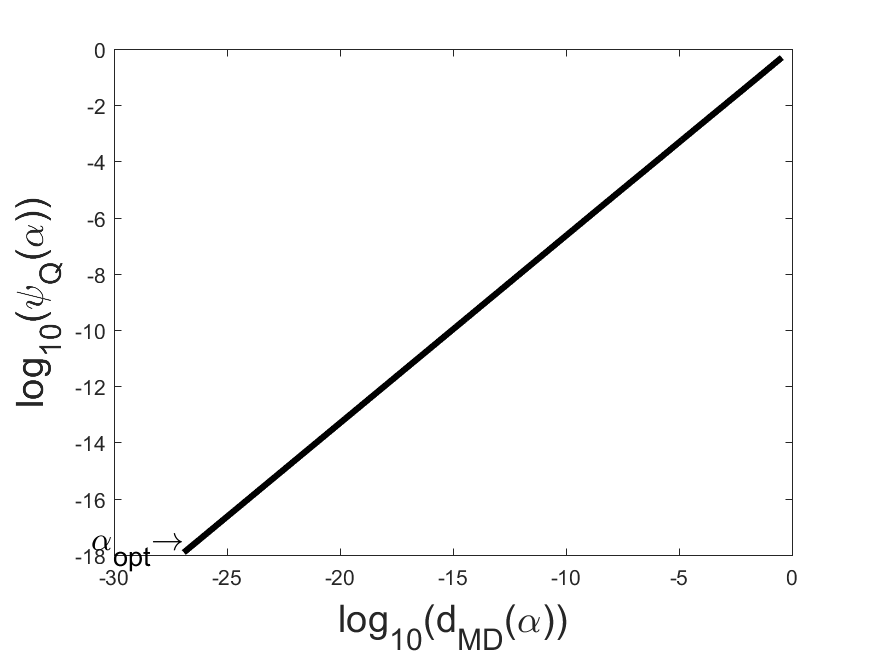}
    \caption{Q-curve in \emph{groetsch2}.}
  \end{minipage}
\end{figure}
We present on the Figures 5-7 examples of Q-curves and triangle $T(k,r(k),l(k))$ with largest area, the TA-rule chooses for the regularization parameter corresponding minimizer. In some problems the function  $\psi_{\mathrm{Q}}(\alpha)$ may be monotonically increasing as for problem \emph{groetsch2} (Figure 8), then the function $\psi_{\mathrm{Q}}(\alpha)$ has only one local minimizer $\alpha_N$. Then vertices $P(M_{l(k)})$ and $P(m_k)$ coincide and area of corresponding triangle is zero. Then this is the only triangle, the TA-rule chooses for the regularization parameter $\alpha_N$. 
The results of the numerical experiments for test set 1 ( $n=100$ ) for the TA-rule and some other rules (see Section 2.2) are given in Tables 3 and 4. These results show that the TA-rule works well in all these test problems, the accuracy is comparable with $\delta$-rules (see Table 2), but previous heuristic rules fail in some problems. Note that average of the error ratio increases for decreasing noise level. For example, for $\n{f-f_*} \in \{10^{-1},10^{-2},10^{-5},10^{-6}\}$ corresponding error ratios $E$ were 1.47, 1.49, 1.78 and 2.08 respectively.

 \begin{table}
\caption{Averages of error ratios E and failure \% (in parenthesis) for heuristic rules}
\label{tab:3}       
\begin{tabular}{lcc|ccccc}
\br
Problem	 & \multicolumn{2}{c|}{TA rule} &Quasiopt.	&WQ	&HR &Reginska	&MCurv\\
  &Mean E & Max E	&Mean E	&Mean E	&Mean E&Mean E	&Mean E\\
\mr
Baart		&1.51  &14.57	&1.54	&1.43	&2.58	&1.32	&4.75\\
Deriv2		&1.18&	1.27&2.01	&2.26	&2.28	&3.67	&(9.2\%)\\
Foxgood		&1.56&	3.39&1.57	&1.57	&8.36	&(10.8\%)	&5.95\\
Gravity		&1.14&	2.27&1.13	&1.13	&2.66	&(0.8\%)	&2.04\\
Heat		&1.26&1.34	&(65.8\%)	&(66.8\%)	&1.64	&(4.2\%)	&4.11\\
Ilaplace	&1.24&2.34	&1.24	&1.22	&1.94	&1.66	&2.99\\
Phillips	&1.07&1.20	&1.09	&(3.3\%)	&2.27	&(44.2\%)	&1.34\\
Shaw		&1.42&	8.96&1.43	&1.41	&2.34	&1.80	&4.64\\
Spikes		&1.01&	5.75&1.01	&1.01	&1.03	&1.01	&1.05\\
Wing		&1.39&6.63	&1.40	&1.30	&1.51	&1.18	&1.57\\
Baker		&3.30&	11.33&3.30	&3.30	&(0.8\%)	&(21.7\%)	&7.78\\
Ursell		&2.87&	31.06&3.54	&2.35	&4.71	&1.86	&7.54\\
Indramm		&3.74&9.07	&4.43	&4.16	&(2.5\%)	&(9.2\%)	&(15.8\%)\\
Waswaz2		&2.43&9.01	&2.43	&2.43	&(65.8\%)	&2.33	&(3.3\%)\\
Groetsch1	&1.14&	4.56&1.14	&1.12	&1.61	&1.26	&1.52\\
Groetsch2	&1.13&1.74	&1.27	&2.73	&1.66	&5.49	&1.81\\
\mr
Total	&1.71	& 31.06&$>100$	&$>100$	&50.5	&43.8	&8.17\\
Failure \%     &0\%	& &4.11\%  &4.38\%  &4.32\%  &5.68\%  &1.77\% \\
Max E2          &2.61  & &$>100$       &$>100$       &2.63   &$>100$       &24.5\\ 
\br
\end{tabular}
\end{table}

Let us comment other heuristic rules.
The accuracy of the quasi-optimality criterion is for many problems the same as for the TA-rule, but this rule fails in problem \emph{heat}. Characteristic feature of the problem \emph{heat} is that location of the eigenvalues in the interval $[\alpha_N,1]$ is sparse and only some eigenvalues are smaller than $\alpha_N$ (see Table 1).
The weighted quasioptimality criterion behaves in a similar way as the quasioptimality criterion, but is more accurate  in problems where $\lambda_{\mathrm{min}} \leq \alpha_N$; if $\lambda_{\mathrm{min}} > \alpha_N$, the quasioptimality criterion is more accurate. 
The rule of Hanke-Raus may fail in test problems with large $\Lambda$  and for other problems the error of the approximate solution is in most problems approximately two times larger than for parameter chosen by the quasi-optimality principle. The problem in this rule is that it chooses too large parameter compared with the optimal parameter. However, HR-rule is stable in the sense that the largest error ratio E2 is relatively small in all considered test problems.
Reginska's rule may fail in many problems but it has the advantage that it works better than other previous rules if the noise level is large. The Reginska's rule 
did not fail in case
$\n{f-f_*} \geq 10^{-3}$ and has average of error ratios of all problems $E=2.24$ and   $E=2.80$ in cases $\n{f-f_*}=10^{-1}$ and $\n{f-f_*}=10^{-2}$ respectively. 
Advantage of the maximum curvature rule is the small percentage of failures compared with other previous rules. 

Distribution of error ratios E in Table 4 shows also that in test problems set 1 from considered rules the TA-rule is the most accurate rule.

\begin{table}
\caption{Distribution of error ratios E in different rules}
\label{tab:4}      
\begin{tabular}{lccccccccc}
\br
Decile	 &TA rule	&Quasiopt.	&WQ	&HR &Reginska	&MCurv & ME & MEe & DP\\
\mr
10	&1.00	&1.00	&1.00	&1.08	&1.00	&1.06	&1.01	&1.00	&1.00\\
20	&1.01	&1.01	&1.01	&1.36	&1.04	&1.27	&1.03	&1.00	&1.01\\
30	&1.02	&1.03	&1.03	&1.56	&1.12	&1.48	&1.09	&1.01	&1.02\\
40	&1.04	&1.06	&1.06	&1.82	&1.27	&1.83	&1.16	&1.03	&1.04\\
50	&1.09	&1.13	&1.12	&2.12	&1.66	&2.31	&1.22	&1.08	&1.08\\
60	&1.18	&1.29	&1.29	&2.43	&2.42	&3.05	&1.33	&1.16	&1.16\\
70	&1.35	&1.57	&1.59	&3.19	&4.19	&4.51	&1.52	&1.29	&1.30\\
80	&1.71	&2.17	&2.29	&5.94	&9.93	&7.03	&2.02	&1.50	&1.52\\
90	&2.27	&6.45	&6.18	&19.35	&43.91	&12.95	&4.45	&2.88	&2.11\\
\br
\end{tabular}
\end{table}

Note that figure of the Q-curve enables to estimate the reliability of chosen parameter. If the Q-curve has only one "corner", then chosen parameter is quasioptimal with small constant $C$, if $\lambda_{\mathrm{min}} < \alpha_N $, but in case $\lambda_{\mathrm{min}} \geq \alpha_N $ it is quasioptimal under assumption that the problem needs regularization.  

\section {Further developments of the area rule}

The TA-rule may fail for problems which do not need regularization, if the function $\psi_{\mathrm{Q}}(\alpha)$ is not monotonically increasing. In this case the TA-rule chooses parameter $\alpha \geq \lambda_{\mathrm{min}}$, but parameter $\alpha < \lambda_{\mathrm{min}}$ would be better. For example, the TA-rule fails for matrix \emph{Moler} in some cases.
Let us consider now the question, in which cases regularization parameter $\alpha_N$ is good. If the function $\psi_{\mathrm{Q}}(\alpha)$ is monotonically increasing then the function $\psi_{\mathrm{Q}}(\alpha)$ has only one local minimizer $m_1=M_1=\alpha_N$ and then for parameter $\alpha_N$ we have the error estimate
 \[ \n{u_{\alpha_N}-u_*} \leq q^{-1} (1+T(\alpha_N)) \min_{\alpha_N \leq \alpha \leq \alpha_0} e_1(\alpha),\] 
where value of $T(\alpha_N)= \max_{\alpha_j  \in \Omega, \alpha_N \leq \alpha_j \leq \alpha_0} T\s{\alpha_N, \alpha_j} \leq c_q \ln(\frac{\alpha_0}{\alpha_N})$ (see Theorem 3) can be computed a posteriori and this value is the smaller the faster the function $\psi_{\mathrm{Q}}(\alpha)$  increases.
We can take $\alpha_N$ for the regularization parameter also in the case if the condition 

 \begin{equation} \label{Gdef}
\frac{\psi_{\mathrm{Q}}(\alpha')}{\psi_{\mathrm{Q}}(\alpha)} \leq c_0 \qquad \forall \alpha, \alpha' \in \Omega, \quad \alpha_N \leq \alpha' < \alpha \leq \alpha_0 
\end{equation}
holds while one can show similarly to the proof of Theorem 3 that $T(\alpha_N) \leq c_0 c_q \ln(\frac{\alpha_0}{\alpha_n})$ and the error of the regularized solution is small.
For problems which do not need regularization we can improve the performance of the TA-rule searching proper local minimizer smaller or equal than 
$\alpha_{\mathrm{HQ}}:=\max\{\alpha_{\mathrm{HR}}, \alpha_{\mathrm{Q}}\},$ where 
$\alpha_{\mathrm{HR}}$, $\alpha_{\mathrm{Q}}$ are global minimizers of functions $\psi_{\mathrm{HR}}(\alpha)$ and $\psi_{\mathrm{Q}}(\alpha)$ respectively on the interval 
$[\max{\s{\alpha_N,\lambda_{\mathrm{min}}}},\alpha_0]$.

These ideas enable to formulate the following upgraded version of the TA-rule.

\textbf{Triangle area rule 2 (TA-2-rule).} We fix a constant
 $c_0, 1 \leq c_0 \leq 2$.  If  
condition (\ref{Gdef}) holds,
 we choose parameter $\alpha_N$. 
Otherwise choose for the regularization parameter such local minimizer $m_k \leq \alpha_{\mathrm{HQ}}$ of the function $\psi_{\mathrm{Q}}(\alpha)$ for which the area of triangle $T(k,r(k),l(k))$ is largest.

Results of numerical experiments for the rule TA-2 with the discretization parameter $n=100$ and problem sets 1-3 are given in Tables 5 and 6 (columns 2 and 3). The results show that rule TA-2 works well in all considered testsets 1-3. However, rule TA-2 may fail in some other problems which do not need regularization. Such example is problem with matrix \emph{moler} and solution $ x_i = \sin{(12 \pi (i -1)/n)}$, where the rule TA-2 fails if the noise level is below $10^{-4}$; but in this case all other considered heuristic rules fail too. 

Rules TA and TA-2 fail in problem \emph{heat} in some cases for discretization parameter $n=60$. Figures 9, 10 show the form of Q-curve in problem \emph{heat} with $n=60$. Function $\psi_{\mathrm{Q}}(\alpha)$ has two local minimizers with corresponding points $P(m_1)$ and $P(m_2)$ on the Q-curve and 3 local maximum points $P(M_k), k=0,1,2$. On Figure 9 Rule TA-2 chooses local minimizer corresponding to the point $P(m_2)$, but then the error ratios are large: $E=20.3, E2=18.34$.  
\begin{figure}[!tbp]
  \centering
  \begin{minipage}[b]{0.45\textwidth}
    \includegraphics[width=\textwidth]
{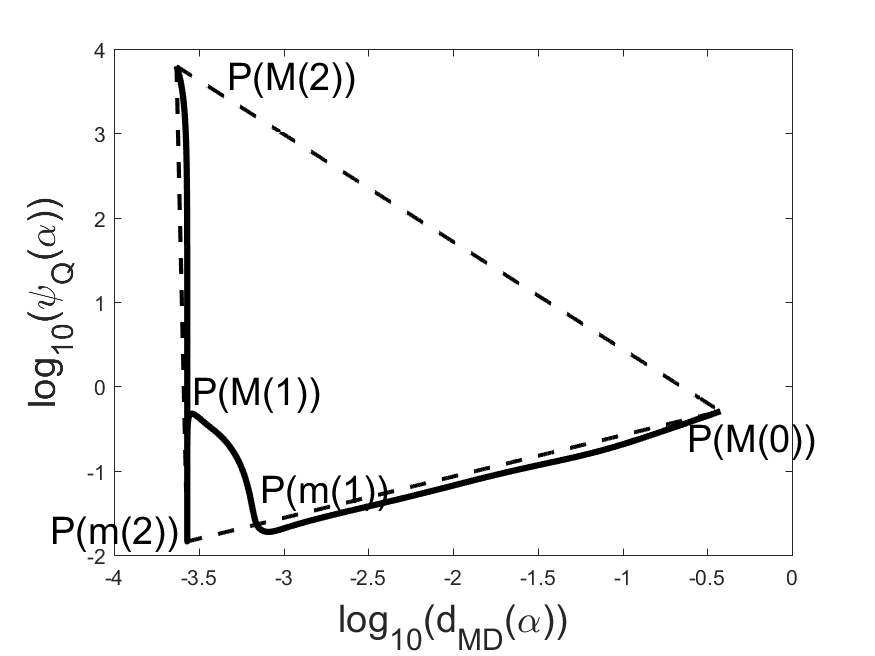}
    \caption{Q-curve, problem  \emph{heat}, $n=60$.}
  \end{minipage}
  \hfill
  \begin{minipage}[b]{0.45\textwidth}
    \includegraphics[width=\textwidth]{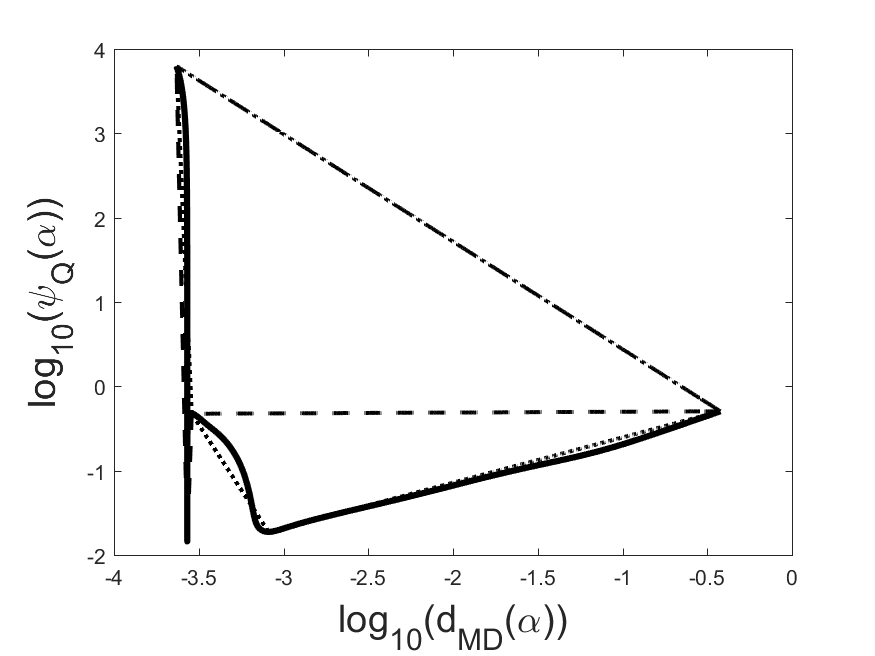}
    \caption{Q-curve, problem \emph{heat}, $n=60$.}
  \end{minipage}
\end{figure}

In the following we consider methods which work well also in this problem.
Let $g[\alpha_1,\alpha_2](\alpha), \, \alpha \in [\alpha_1, \alpha_2]$ be parametric representation of straight line segment connecting points $P(\alpha_1)$ ja  $P(\alpha_2)$, thus 
\[ \fl \qquad g[\alpha_1,\alpha_2](\alpha)=\tilde \psi_{\mathrm{Q}}(\alpha_1)+\beta (\tilde d_{\mathrm{MD}}(\alpha)-\tilde d_{\mathrm{MD}}(\alpha_1)), \quad
\beta=\frac{\tilde \psi_{\mathrm{Q}}(\alpha_2)-\tilde \psi_{\mathrm{Q}}(\alpha_1)}{\tilde d_{\mathrm{MD}}(\alpha_2)-\tilde d_{\mathrm{MD}}(\alpha_1)}.\]
Let $g[\alpha_1,\alpha_2,...,\alpha_k](\alpha)$, $k>2$  be parametric representation of broken line connecting points $P(\alpha_1)$, $P(\alpha_2),\ldots,P(\alpha_k)$, thus  
\[\fl \qquad g[\alpha_1,\alpha_2,...,\alpha_k](\alpha)=g[\alpha_j,\alpha_{j+1}](\alpha), \qquad  \alpha_j \leq \alpha \leq \alpha_{j+1}, \quad 1\leq j \leq k-1.\]
In triangle rule certain points $P(M_{l(k)}),P(m_k),P(M_{r(k)})$ are connected by the broken line $t_1(\alpha)=g[M_{l(k)},m_k,M_{r(k)}](\alpha)$ which approximates the error function $\tilde e_1(\alpha)=\log_{10}(e_1(\alpha))$ well, if $m_k$ is the "right" local minimizer. By construction of function $t_1(\alpha)$ we use only 3 points on the Q-curve. We will get a more stable rule if the form of the Q-curve has more influence  to the construction of approximates to the error function $\tilde e_1(\alpha)$.
Let $\{i(1),i(2),...,i(n1)\}$ and $\{j(1),j(2),...,j(n2)\}$ be the largest sets of indices, satisfying the inequalities
\[ \fl \qquad k \leq i(1)<i(2)<...<i(n1) \leq K, \quad \psi_{\mathrm{Q}}(M_{i(1)}) \leq \psi_{\mathrm{Q}}(M_{i(2)})  \leq ... \leq \psi_{\mathrm{Q}}(M_{i(n1)}), \]
\[ \fl \qquad k > j(1)>j(2)>...>j(n2) \geq 0, \quad \psi_{\mathrm{Q}}(M_{j(1)}) \leq \psi_{\mathrm{Q}}(M_{j(2)})  \leq ... \leq \psi_{\mathrm{Q}}(M_{j(n2)}). \] 
It is easy to see that $i(n1)=l(k)$ and $j(n2)=r(k)$.  For approximating the error function $\tilde e_1(\alpha)$ we propose to connect points $P(M_{i(n1)}),...,P(M_{i(1)}),P(m_k), P(M_{j(1)}),...,P(M_{j(n2)}) $ by broken line  $t_2(\alpha)=g[M_{i(n1)},...,M_{i(1)},m_k,M_{j(1)},...,M_{j(n2)}](\alpha)$  and to find for every $m_k$ the area $S_2(k)$ of polygon surrounded by lines 
$T_2(\alpha)=\max \{t_2(\alpha),g[M_{i(n1)},M_{j(n2)}](\alpha)\}$ and $t_2(\alpha)$.
The second possibility is to approximate the error function $\tilde e_1(\alpha)$ by the curve  $t_3(\alpha)=\max\{t_2(\alpha), \tilde \psi_{\mathrm{Q}}(\alpha)\}$  and to find  $S_3(k)$ as the area of polygon surrounded by broken lines $t_3(\alpha)$ and curve  $T_3(\alpha)=\max\{T_2(\alpha), \tilde \psi_{\mathrm{Q}}(\alpha)\}$.  
Note that functions $t_i(\alpha), i=1,2,3$ are monotonically increasing if $\alpha>m_k$, and monotonically decreasing if $\alpha<m_k$. 

\textbf{Area rules 2 and 3.} We fix constant $c_0, 1 \leq c_0 \leq 2$.  First we choose local minimizer  $ m_k \leq \alpha_{HQ}$, for which the area $S_i(k), \ i \in \{2,3\}$ is largest. We take for the regularization parameter the smallest  $ m_{k_0}  \leq m_k$, satisfying the condition (compare with (\ref{Gdef})) 
\[\frac{\psi_{\mathrm{Q}}(\alpha')}{\psi_{\mathrm{Q}}(\alpha)} \leq c_0 \qquad \forall \alpha, \alpha' \in \Omega, \quad m_{k_0} \leq \alpha' < \alpha \leq m_k. \] 
 
Let us consider Figure 9. The reason of failure of triangle angle rule is, that for local minimizer $m(2)$ the broken line $g[M(0),m(2),M(2)](\alpha)$ do not approximate well the function  $\tilde e_1(\alpha)$, while point $M(1)$ is located above the interval $[m(2),M(0)]$. Here the function $\tilde e_1(\alpha)$ is better approximated by the broken line $g[M(0),M(1),m(2),M(2)](\alpha)$, see Figure 10. For local minimizer $m(1)$ we approximate function $\tilde e_1(\alpha)$ by broken line  $g[M(0),m(1),M(1),M(2)](\alpha)$ and due to the inequality $S_2(1)>S_2(2)$ rule 2 chooses $m_1$ for the regularization parameter, then $E=1.05$. 

Area rules 2 and 3 work in problem \emph{heat} well for every $n$ and all $\alpha_N =10^{-k}, 12 \leq k \leq 24 $, but in some other problems accuracy of area rules 2 and 3 (see columns 4-7 of Tables 5, 6) is slightly worser than for rule TA-2. The advantage of area rule 3, as compared to area rule 2, is to be highlighted in problem \emph{heat} if all noise of the right hand side is placed on one eigenelement (then we use the condition $ m_k \leq \alpha_{HQ}$ only in case $\lambda_{\mathrm{min}}>\alpha_N$). Then area rule 3 did not fail if $n \geq 80$ and $\alpha_N  \leq 10^{-20}$. So we can say, the more precisely we take into account the form of the Q-curve in construction of the approximating function for the error function $\tilde e_1(\alpha)$, the more stable is the rule.

\begin{table}
\caption{Averages and maximums of error ratios E in case of area rules, problem set 1}
\label{tab:5}       
\begin{tabular}{l|cc|cc|cc|cc}
\br
Problem  & \multicolumn{2}{c|}{TA-2 rule} &\multicolumn{2}{c|}{Area rule 2} & \multicolumn{2}{c|}{Area rule 3}  & \multicolumn{2}{c}{Combined area rule}\\
 & Aver E & Max E  & Aver E & Max E & Aver E & Max E & Aver E & Max E\\
\mr
Baart       & 1.51 & 5.18 & 1.58 & 2.91  & 1.59 & 2.91  & 1.53 & 5.18\\
Deriv2     & 1.12 & 1.42 & 1.12 & 1.42  & 1.12 & 1.42  & 1.12 & 1.42  \\
Foxgood   & 1.57 & 6.69 & 1.53 & 6.19  & 1.53 & 6.19  & 1.57 & 6.69\\
Gravity     & 1.17 & 4.12 & 1.21 & 6.10 & 1.21 & 6.10  & 1.17 & 4.12\\
Heat        & 1.12 & 2.36 & 1.12 & 2.36 & 1.12 & 2.36 & 1.12 & 2.36  \\
Ilaplace    & 1.24 & 2.68 & 1.22 & 2.68 & 1.22 & 2.68 & 1.24 & 2.68\\
Phillips      & 1.07 & 1.72 & 1.06 & 1.72 & 1.06 & 1.72 & 1.07 & 1.72 \\
Shaw       & 1.42 & 3.72 & 1.47 & 3.64 & 1.47 & 3.64 & 1.42 & 3.72\\
Spikes      & 1.01 & 1.05 & 1.01 & 1.02 & 1.01 & 1.02 & 1.01 & 1.05 \\
Wing        & 1.39 & 1.86 & 1.44 & 1.86 & 1.44 & 1.86 & 1.39 & 1.86\\
Baker        & 3.30 & 45.29 & 2.67 & 22.67 & 2.67 & 22.67 & 2.91 & 33.12 \\
Ursell        & 2.87 & 16.78 & 4.55 & 27.92 & 4.55 & 27.92 & 3.12 & 16.78\\
Indramm   & 3.74 & 25.67 & 9.50 & 83.20 & 10.76 & 83.20 & 3.87 & 25.67 \\
Waswaz2    & 2.43 & 9.01 & 2.43 & 9.01 & 2.43 & 9.01 & 2.43 & 9.01\\
Groetsch1 & 1.14 & 2.12 & 1.15 & 2.12 & 1.15 & 2.12 & 1.14 & 2.12 \\
Groetsch2 & 1.52 & 3.84 & 1.52 & 3.84 & 1.52 & 3.84 & 1.52 & 3.84\\
\mr
Total  & 1.73 & 45.29 & 2.16 & 83.20  & 2.22 & 83.20  & 1.73 & 33.12\\
\br
\end{tabular}
\end{table}
Based on the above rules it is possible to formulate a combined rule, which chooses the parameter according to the rule TA-2 or area rule 3 in dependence of certain condition.

\textbf{Area rule 4 (Combined area rule).} Fix constant $c_0, 1 \leq c_0 \leq 2,  b \geq 0$.   Let local minimizer $m_k$ be chosen by the rule TA-2.  If 
\[\max_{m_k \leq \alpha \leq M_{r(k)}} \frac{\tilde \psi_{\mathrm{Q}}(\alpha)}{g[m_k,M_{r(k)}](\alpha)} \leq b, \]
we take $m_k$ for the regularization parameter, otherwise we choose regularization parameter by rule 3. 

Note that combined rule coincides with rule TA-2, if  $b=0$ and with area rule 3, if  $b=\infty$. Experiments of combined rule with $c_0=2, b=1$ (columns 8 and 9 in Tables 5, 6)) show that accuracy of this rule is  almost the same as in triangle rule, but unlike the TA-2 rule, it works well also in the problem \emph{heat} for all $n$ and $\alpha_N$. Although, in some cases, in test set 3 the error ratio $E>100$ for rule 4,  the high qualification of the rule is characterized by fact, that over all problems sets 1-3  the largest error ratio E1 was 16.91
(5.06 for set 1) and the largest error ratio E2 was 4.67 (2.62 for set 1). Numerical experiments show that it is reasonable to use parameter $b \in( 0.8,1.2]$. 
We studied the behavior of area rules for different $\alpha_N=10^{-k}, 12 \leq k \leq 24 $. The results were similar to results of Tables 5 and 6, but for smaller  $\alpha_N$ the error ratios were 2-3\% smaller than for  $\alpha_N=10^{-18}$ and for larger  $\alpha_N$ the error ratios were about 5\% larger than in Tables 5, 6. 

\begin{table}
\caption{Averages and maximums of error ratios E in proposed rules, problem sets 2 and 3.} 
\label{tab:6}       
\begin{tabular}{l|cc|cc|cc|cc}
\br
Problem  & \multicolumn{2}{c|}{TA-2 rule} &\multicolumn{2}{c|}{Area rule 2} & \multicolumn{2}{c|}{Area rule 3}  & \multicolumn{2}{c}{Combined area rule}\\
 & Aver E & Max E  & Aver E & Max E & Aver E & Max E & Aver E & Max E\\
\mr
Gauss    & 1.24 & 5.05 & 1.26 & 6.56  & 1.26 & 6.56  & 1.24 & 5.05\\
Hilbert   & 1.46 & 7.25 & 1.83 & 21.22  & 1.81 & 21.22  & 1.46 & 7.25  \\
Lotkin    & 1.47 & 11.17 & 1.91 & 18.66  & 1.88 & 11.17  & 1.47 & 11.17\\
Moler     & 1.51 & 7.35 & 1.43 & 7.35 & 1.43 & 7.35  & 1.51 & 7.35\\
Prolate   & 1.57 & 15.96 & 1.82 & 20.64 & 1.77 & 15.96 & 1.58 & 15.96  \\
Pascal    & 1.04 & 1.13 & 1.06 & 1.18 & 1.06 & 1.18 & 1.05 & 1.18\\
\mr
Set 2  & 1.38 & 15.96 & 1.55 & 21.22 & 1.53 & 21.22  & 1.39 & 15.96\\
Set 3  & 1.85 & 136.6 & 2.81 & 188.1 & 2.77 & 188.1  & 2.02 & 153.5\\
\br
\end{tabular}
\end{table}

The Table 7 gives results of the numerical experiments in the case of smooth solution,
$p=2$.  We see that combined rule worked well also in this case, no failure.
\begin{table}
\caption{Results of the numerical experiments, $p=2$}
\label{tab:7}       
\begin{tabular}{lc|c|c|c|c|cc}
\br
Problem & ME & MEe & DP  & Best of $L_{\mathrm{min}}$ &  $|L_{\mathrm{min}}|$  & \multicolumn{2}{c}{Combined area rule}\\
 & Aver E & Aver E & Aver E & Aver E  & Aver & Aver E & Max E \\
\mr
Baart       & 1.86 & 1.19  & 2.93 & 1.18  & 4.74 & 1.60 & 14.57\\
Deriv2      & 1.09 & 1.19 &  3.65 & 1.03  & 2.00 & 1.04 & 1.17\\
Foxgood    & 1.56 & 1.13 &  3.58 & 1.14  & 2.08 & 1.22 & 3.58\\
Gravity      & 1.33 & 1.05  &  2.65 & 1.09 & 1.72 & 1.14 & 3.18\\
Heat         & 1.13 & 1.12 & 2.55 & 1.05 & 2.10 & 1.05 & 1.14\\
Ilaplace     & 1.47 & 1.06 & 2.78 & 1.11  & 2.73 & 1.13 & 3.51\\
Phillips       & 1.26 & 1.06 & 3.35 & 1.04  & 2.10 & 1.04 & 1.20\\
Shaw         & 1.37 & 1.06 & 2.58 & 1.11  & 3.72 & 1.29 & 8.96\\
Spikes        & 1.85 & 1.12  & 2.10 & 1.19  & 4.78 & 1.31 & 5.75 \\
Wing          & 1.67 & 1.14 & 2.47 & 1.22  & 4.53 & 1.75 & 6.63\\
Baker          & 2.11 & 1.29 & 2.96 & 1.21 & 4.38 & 1.77 & 11.33\\
Ursell          & 1.86 & 1.19 & 4.10 & 1.16  & 4.82 & 1.67 & 18.08\\
Indramm     & 1.69 & 1.14 & 2.87 & 1.28  & 4.53  & 1.91 &6.42\\
Waswaz2     & 127.2 & 49.8 & 1.20 & 2.44  & 1.00 & 2.43 & 9.01\\
Groetsch1  & 1.40 & 1.06  & 2.36 & 1.11 & 2.14 & 1.14 & 4.56\\
Groetsch2  & 1.02 & 1.23 & 1.71 & 1.14 & 1.67 & 1.55 & 3.97 \\
\mr
Set 1   & 9.37 & 4.18 & 2.74 & 1.22 & 3.06 & 1.44 & 18.08\\
Set 2   & 2.10 & 1.26 & 2.91 & 1.19  & 2.83 &1.37 & 29.03\\
Set 3   & 6.86 & 3.21 & 2.68 & 1.18 & 3.12 & 1.42 & 52.98\\
\br
\end{tabular}
\end{table}

\begin{remark}
It is possible to modify the Q-curve. We may use the function $\psi_{\mathrm{QD}}(\alpha)$ instead of function $\psi_{\mathrm{Q}}(\alpha)$ and find proper local minimizer of the function $\psi_{\mathrm{QD}}(\alpha)$.
Unlike the quasi-optimality criterion the use of function $\psi_{\mathrm{QD}}(\alpha)$ in the Q-curve and in the area rule does not increase the amount of calculations, while approximation $u_{2,\alpha}$ is needed also in computation of $d_{\mathrm{MD}}(\alpha)$. We can use in these rules the function $d_{\mathrm{ME}}(\alpha)$ instead of $d_{\mathrm{MD}}(\alpha)$, it increases the accuracy in some problems, but the average accuracy of the rules is almost the same.
In case of nonsmooth solutions we can modify the Q-curve method and area rule, using the function $d_{D}(\alpha)$ instead of $d_{\mathrm{MD}}(\alpha)$. In this case, we get even better results for $ p = 0 $ but for $ p = 2 $, the error ratio $ E $ is on average 2 times higher.

Note that if  solution is smooth, then L-curve rule and Reginska's rule often fail, but replacing in these rules 
the functions $Au_\alpha-f$ and $u_\alpha=-\alpha^{-1}A^*(Au_\alpha-f)$ by functions  $B_\alpha(Au_\alpha-f)$ and 
$-\alpha^{-1}A^*B^2_\alpha (Au_\alpha-f)$ (then Reginska's rule modifies to minimization of the function $\psi_{\mathrm{WQ}}(\alpha)$ (see (\ref{WQ})) respectively gives often better results.
\end{remark} 

In the case of a heuristic parameter choice, it is also possible to use the a posteriori estimates of the approximate solution, which, in many tasks, allows
to confirm the reliability of the parameter choice. Let $\alpha_H$ be the regularization parameter from some heuristic rule and $\alpha_*$ be the local minimizer of the function $e_1(\alpha)$ on the set $\Omega$. Then in case $\alpha_* \geq \alpha_H$ the error estimate
\begin{equation}  \label{eq22}
\n{u_{\alpha_H}-u_*} \leq \s{1+T(\alpha_H, \alpha_*)}e_1(\alpha_*) \leq q^{-1}\s{1+T_1(\alpha_H)} \min_{\alpha} e_1(\alpha)
\end{equation} 
holds where $T_1(\alpha_H)=\max_{\alpha \geq \alpha_H, \alpha \in \Omega} T(\alpha_H, \alpha)$.
Using the last estimate, we can prove  similarly to the Theorem 7 that if $\alpha_N$
is so small that $d_{\mathrm{MD}}(\alpha_N) \leq (1+\epsilon) \n{f-f_*}$, then
\[  \n{u_{\alpha_H}-u_*} \leq \max\{q^{-1}(1+T_1(\alpha_H)) \min_{\alpha \geq 0} e_1(\alpha), C_2(b,\epsilon) \min_{\alpha \geq 0}  e_2(\alpha,\n{f-f_*})\}, \]
where $b=d_{\mathrm{MD}}(\alpha_H)/d_{\mathrm{MD}}(\alpha_N)$. If values $T_1(\alpha_H)$ and $b$ what we find a posteriori, are small (for example $b \leq 2$ and $T_1(\alpha_H) \leq 9$), then this estimate allows to argue that error of approximate solution for this parameter is not much larger than the minimal error.
The conditions $b \leq 2$ , $T_1(\alpha_H) \leq 9$ were satisfied in set 1 of test problems in combined rule for 73\% of cases and inequalities $b \leq 2$ , $T_1(\alpha_H) \leq 4$ for 61\% of cases. The reason of failure of heuristic rule is typically that chosen parameter is too small.  To check this, we can use the error estimate (\ref {eq22}). If $T_1(\alpha_H)$ is relatively small (for example $T_1(\alpha_H) \leq 9$), then estimate (\ref{eq22}) allows to argue that the regularization parameter is not chosen too small. In set 1 of test problems the conditions  $T_1(\alpha_H) \leq 9$ and $T_1(\alpha_H) \leq 4$ were satisfied in 97\% and in  82\% of cases respectively. 

We finish the paper with the following conclusion. For the heuristic choice of the regularization parameter we recommend to choose the parameter from the set of local minimizers of the function $\psi_{\mathrm{Q}}(\alpha)$ or the function  $\psi_{\mathrm{QD}}(\alpha)$. For choice of the parameter from the local minimizers we proposed the Q-curve method and different area rules. The proposed rules gave  much better results than  previous heuristic rules on extensive set of test problems. Area rules fail in very few cases in comparison with previous rules, and the accuracy of these rules is comparable even with the $\delta$-rules  if the exact noise level is known. In addition, we also provided a posteriori error estimates of the approximate solution, which allows to check the reliability of parameter chosen heuristically.

\section*{Acknowledgment}

The authors are supported 
by institutional research funding IUT20-57
of the Estonian Ministry of Education and Research.

\end{document}